\documentclass[11pt]{amsart}

\newtheorem{theorem}{Theorem}[section]

\newtheorem{corollary}[theorem]{Corollary}

\newtheorem{definition}[theorem]{Definition}
\newtheorem{lemma}[theorem]{Lemma}

\newtheorem{proposition}[theorem]{Proposition}
\newtheorem{remark}[theorem]{Remark}

\usepackage[colorlinks, bookmarks=true]{hyperref}
\usepackage{color,graphicx,shortvrb}
\usepackage[active]{srcltx} 

\def\11{\textbf{$1$}}

\usepackage{enumerate}
\usepackage{amsmath}
\usepackage{amssymb}
\usepackage[latin 1]{inputenc}

\begin{document}

\title{Weak 2-local derivations on $\mathbb{M}_n$}

\author[Niazi]{Mohsen Niazi}
\address{Department of Mathematics, Faculty of Mathematical and Statistical Sciences, University of Birjand, Birjand, Iran.}
\curraddr{Departamento de An{\'a}lisis Matem{\'a}tico, Facultad de
Ciencias, Universidad de Granada, 18071 Granada, Spain.}
\email{niazi@birjand.ac.ir}

\author[Peralta]{Antonio M. Peralta}
\address{Departamento de An{\'a}lisis Matem{\'a}tico, Facultad de
Ciencias, Universidad de Granada, 18071 Granada, Spain.}
\curraddr{Visiting Professor at Department of Mathematics, College of Science, King Saud University, P.O.Box 2455-5, Riyadh-11451, Kingdom of Saudi Arabia.}
\email{aperalta@ugr.es}

\thanks{Authors partially supported by the Spanish Ministry of Science and Innovation, D.G.I. project no. MTM2011-23843, and Junta de Andaluc\'{\i}a grant FQM375. Second author partially supported by the Deanship of Scientific Research at King Saud University (Saudi Arabia) research group no. RG-1435-020. The first author acknowledges the partial financial support from the IEMath-GR program for visits of young talented researchers.}



\begin{abstract} We introduce the notion of weak-2-local derivation (respectively, $^*$-derivation) on a C$^*$-algebra $A$ as a (non-necessarily linear) map  $\Delta : A\to A$ satisfying that for every $a,b\in A$ and $\phi\in A^*$ there exists a derivation (respectively, a $^*$-derivation) $D_{a,b,\phi}: A\to A$, depending on $a$, $b$ and $\phi$, such that $\phi \Delta (a) = \phi D_{a,b,\phi} (a)$ and $\phi \Delta (b) = \phi D_{a,b,\phi} (b)$. We prove that every weak-2-local $^*$-derivation on $M_n$ is a linear derivation. We also show that the same conclusion remains true for weak-2-local $^*$-derivations on finite dimensional C$^*$-algebras.
\end{abstract}


\maketitle
 \thispagestyle{empty}

\section{Introduction and preliminaries}

``Derivations appeared for the first time at a fairly early stage in the young
field of C$^*$-algebras, and their study continues to be one of the central branches in the field'' (S. Sakai, 1991 \cite[Preface]{Sak91}). We recall that \emph{derivation} from an associative algebra $A$ into an $A$-bimodule $X$ is a linear mapping $D: A\to X$ satisfying $$D(a b) = D(a) b +a D(b), \ \ (a,b\in A).$$ If $A$ is a C$^*$-algebra and $D$ is a derivation on $A$ satisfying $D(a^*) = D(a)^*$ ($a\in A$), we say that $D$ is \emph{$^*$-derivation} on $A$.\smallskip

Some of the earliest, remarkable contributions on derivations are due to Sakai. For example, a celebrated result due to him shows that every derivation on a C$^*$-algebra is continuous \cite{Sak60}. A subsequent contribution proves that every derivation on a von Neumann algebra $M$ is inner, that is, for every derivation $D$ on $M$ there exists ${a}\in M$ satisfying $D(x)= [{a},x] = {a} x- x {a},$ for
every $x\in M$ (cf. \cite[Theorem 4.1.6]{Sak}).\smallskip

We recall that, accordingly to the definition introduced by R.V. Kadison in \cite{Kad90}, a linear mapping $T$ from a Banach algebra $A$ into a $A$-bimodule $X$ is said to be a \emph{local derivation} if for every $a$ in $A$, there exists a derivation $D_{a}: A\to X$, depending on $a$, such that $T(a)= D_{a} (a)$. The contribution due to Kadison establishes that every continuous local derivation from a von Neumann algebra $M$ into a dual $M$-bimodule $X$ is a derivation. B.E. Johnson proves in \cite{John01} that every local derivation from a C$^*$-algebra $A$ into a Banach $A$-bimodule is a derivation.\smallskip

A very recent contribution, due to A. Ben Ali Essaleh, M.I. Ramírez and the second author of this note, establishes a new characterization of derivations on a C$^*$-algebra $A$, in weaker terms than those in the definition of local derivations given by Kadison (cf. \cite{BenAliPeraltaRamirez}). A linear mapping $T: A\to A$ is a \emph{weak-local derivation} if for every $a\in A$ and every $\phi\in A^{*},$ there exists a derivation $D_{a,\phi}: A\to A$, depending on $a$ and $\phi$, satisfying $\phi T (a) = \phi D_{a,\phi} (a)$ (cf. \cite[Definition 1.1 and page 3]{BenAliPeraltaRamirez}). Theorem 3.4 in \cite{BenAliPeraltaRamirez} shows that every weak-local derivation on a C$^*$-algebra is a derivation.\smallskip

When in the definition of local derivation we relax the condition concerning linearity but we assume \emph{locality} at two points, we find the notion of 2-local derivation introduced by P. \v{S}emrl in \cite{Semrl97}. Let $A$ be a Banach algebra. A (non-necessarily linear) mapping $\Delta: A\to A$ is said to be a \emph{2-local derivation} if for every $a,b\in A$ there exists a derivation $D_{a,b}: A\to A$, depending on $a$ and $b$, satisfying $\Delta (a) = D_{a,b} (a)$ and $\Delta (b) = D_{a,b} (b)$.  \v{S}emrl proves in \cite[Theorem 2]{Semrl97} that for an infinite-dimensional separable Hilbert space $H$, every 2-local derivation on the algebra $B(H)$ of all linear bounded operators on $H$ is linear and a derivation. S.O. Kim and J.S. Kim gave in \cite{KimKim04} a short proof of the fact that every 2-local derivation on $M_n$, the algebra of $n \times n$ matrices over the complex numbers, is a derivation. In a recent contribution, S. Ayupov and K. Kudaybergenov prove that every 2-local derivation on an arbitrary von Neumann algebra is a derivation (see \cite{AyuKuday2014}).\smallskip

In this note we introduce the following new class of mappings on C$^*$-algebras:

\begin{definition}\label{def weak-2-local} Let $A$ be a C$^*$-algebra, a {\rm(}non-necessarily linear{\rm)} mapping $\Delta: A\to A$ is said to be a weak-2-local derivation {\rm(}respectively, a weak-2-local $^*$-derivation{\rm)} on $A$ if for every $a,b\in A$ and $\phi\in A^*$ there exists a derivation {\rm(}respectively, a $^*$-derivation{\rm)} $D_{a,b,\phi}: A\to A$, depending on $a$, $b$ and $\phi$, such that $\phi \Delta (a) = \phi D_{a,b,\phi} (a)$ and $\phi \Delta (b) = \phi D_{a,b,\phi} (b)$.
\end{definition}

The main result of this paper (Theorem \ref{t weak-2-local derivations on Mn are derivations}) establishes that every (non-necessarily linear) weak-2-local $^*$-derivation on $M_n$ is a linear $^*$-derivation. We subsequently prove that every weak-2-local $^*$-derivation on a finite dimensional C$^*$-algebra is a linear $^*$-derivation. These results deepen on our knowledge about derivations on C$^*$-algebras and the excellent behavior that these operators have in the set of all maps on a finite dimensional C$^*$-algebra.\smallskip

As in previous studies on 2-local derivations and $^*$-homomorphisms (cf. \cite{AyuKuday2014, KOPR2014, BurFerGarPe2014RACSAM, BurFerGarPe2014JMAA} and \cite{AyuKudPe2014}), the techniques in this paper rely on the Bunce-Wright-Mackey-Gleason theorem \cite{BuWri92}, however, certain subtle circumstances and pathologies, which are intrinsical to the lattice $\mathcal{P} (M_n)$ of all projections in $M_n,$ increase the difficulties with respect to previous contributions. More concretely, the just mentioned Bunce-Wright-Mackey-Gleason theorem asserts that every bounded, finitely additive (vector) measure on the set of projections of a von Neumann algebra $M$ with no direct summand of Type $I_2$ extends (uniquely) to a bounded linear operator defined on $M$. Subsequent improvements due to S.V. Dorofeev and A.N. Sherstnev establish that every completely additive measure on the set of projections of a von Neumann algebra with no type $I_n$ ($n<\infty$) direct summands is bounded (\cite{Doro,Shers2008}). In the case of $M_n$, there exist completely additive measures on $\mathcal{P} (M_n)$ which are unbounded (see Remark \ref{r non boundedness of completely additive measure on PMn}). We establish a new result on non-commutative measure theory by proving that every weak-2-local $^*$-derivation on $M_n$ ($n\in \mathbb{N}$) is bounded on the set $\mathcal{P}(M_n)$ (see Proposition \ref{p boundedness of Delta}). This result shows that under a weak algebraic hypothesis we obtain an analytic implication, which provides the necessary conditions to apply the Bunce-Wright-Mackey-Gleason theorem.\smallskip

In this paper we also prove that every weak-2-local derivation on $M_2$ is a linear derivation. Numerous topics remain to be studied after this first answers. Weak-2-local derivations on $M_n$ and weak-2-local ($^*$-)derivations on von Neumann algebras and C$^*$-algebras should be examined.

\section{General properties of weak-2-local derivations}

Let $A$ be a C$^*$-algebra. Henceforth, the symbol $A_{sa}$ will denote the self-adjoint part of $A$. It is clear, by the Hahn-Banach theorem, that every weak-2-local derivation $\Delta$ on $A$ is 1-homogeneous, that is, $\Delta (\lambda a) = \lambda \Delta (a),$ for every $\lambda\in \mathbb{C}$, $a\in A.$\smallskip

We observe that the set Der$(A)$, of all derivations on $A$, is a closed subspace of the Banach space $B(A)$. This fact can be applied to show that a mapping $\Delta : A \to A$ is a weak-2-local derivation if and only if for any set $V\subseteq A^*$, whose linear span is $A^*$, the following property holds: for every $a,b\in A$ and $\phi\in V$ there exists a derivation $D_{a,b,\phi}: A\to A$, depending on $a$, $b$ and $\phi$, such that $\phi \Delta (a) = \phi D_{a,b,\phi} (a)$ and $\phi \Delta (b) = \phi D_{a,b,\phi} (b)$. This result guarantees that in Definition \ref{def weak-2-local} the set $A^*$ can be replaced, for example, with the set of positive functionals on $A$.\smallskip

Let $\Delta : A\to A$ be a mapping on a C$^*$-algebra. We define a new mapping $\Delta^{\sharp} : A\to A$ given by $\Delta^{\sharp} (a) := \Delta (a^*)^*$ ($a\in A$). Clearly, $\Delta^{\sharp \sharp} = \Delta$. It is easy to see that $\Delta$ is linear (respectively a derivation) if and only if $\Delta^{\sharp}$ is linear (respectively, a derivation). We also know that $\Delta (A_{sa}) \subseteq A_{sa}$ whenever $\Delta^{\sharp} = \Delta$.\smallskip

Let $A$ be a C$^*$-algebra. A  mapping $\Delta: A\to A$ is said to be a \emph{weak-2-local $^*$-derivation} on $A$ if for every $a,b\in A$ and $\phi\in A^*$ there exists a $^*$-derivation $D_{a,b,\phi}: A\to A$, depending on $a$, $b$ and $\phi$, such that $$\phi \Delta (a) = \phi D_{a,b,\phi} (a) \hbox{ and }\phi \Delta (b) = \phi D_{a,b,\phi} (b).$$

Clearly, every weak-2-local $^*$-derivation $\Delta$ on $A$ is a weak-2-local derivation and $\Delta^{\sharp} = \Delta$. However, we do not know if every weak-2-local derivation with $\Delta^{\sharp} = \Delta$ is a weak-2-local $^*$-derivation. Anyway, for a weak-2-local derivation $\Delta: A\to A$ with $\Delta^{\sharp} = \Delta,$ the mapping $\Delta|_{A_{sa}} : A_{sa} \to A_{sa}$ is a weak-2-local Jordan derivation, that is, for every $a,b\in A_{sa}$ and $\phi\in (A_{sa})^*,$ there exists a Jordan $^*$-derivation $D_{a,b,\phi}: A_{sa}\to A_{sa}$, depending on $a$, $b$ and $\phi$, such that $$\phi \Delta (a) = \phi D_{a,b,\phi} (a) \hbox{ and }\phi \Delta (b) = \phi D_{a,b,\phi} (b).$$ To see this, let $a,b\in A_{sa}$ and $\phi\in (A_{sa})^*$, by assumptions, there exists a derivation $D_{a,b,\phi}: A\to A$, depending on $a$, $b$ and $\phi$, such that $\phi \Delta (a) = \phi D_{a,b,\phi} (a)$ and $\phi \Delta (b) = \phi D_{a,b,\phi} (b)$. Since $\phi \Delta (a) = \phi \Delta (a)^* = \phi D_{a,b,\phi}^{\sharp} (a)$ and $\phi \Delta (b) = \phi D_{a,b,\phi}^{\sharp} (b)$, we get $$\phi \Delta (a) = \phi \frac12 \left(D_{a,b,\phi}-D_{a,b,\phi}^{\sharp}\right) (a), \hbox{ and } \phi \Delta (b) = \phi \frac12 \left(D_{a,b,\phi}-D_{a,b,\phi}^{\sharp}\right) (b),$$ where $\frac12 \left(D_{a,b,\phi}-D_{a,b,\phi}^{\sharp}\right)$ is a $^*$-derivation on $A.$\smallskip

The following properties can be also deduced from the fact stated in the second paragraph of this section.

\begin{lemma}\label{l sharp} Let $A$ be a C$^*$-algebra. The following statements hold:
\begin{enumerate}[$(a)$] \item The linear combination of weak-2-local derivations on $A$ is a weak-2-local derivation on $A$;
\item A mapping $\Delta : A\to A$ is a weak-2-local derivation if and only if $\Delta^{\sharp}$ is a weak-2-local derivation;
\item A mapping $\Delta : A\to A$ is a weak-2-local derivation if and only if $\Delta_s=\frac12 (\Delta+ \Delta^{\sharp})$ and $\Delta_a= \frac{1}{2 i} (\Delta- \Delta^{\sharp})$ are weak-2-local derivations. Clearly, $\Delta$ is linear if and only if both $\Delta_s$ and $\Delta_a$ are.
\end{enumerate}
\end{lemma}

\begin{proof}$(a)$ Suppose $\Delta_1,\ldots, \Delta_n : A\to A$ are weak-2-local derivations and $\lambda_1,\ldots, \lambda_n$ are complex numbers. Given $a,b\in A$ and $\phi \in A^*$, we can find derivations $D^{j}_{a,b,\phi}: A\to A$ satisfying $\phi \Delta_j (a) = \phi D^{j}_{a,b,\phi} (a)$ and $\phi \Delta_j (b) = \phi D^{j}_{a,b,\phi} (b)$, for every $j=1,\ldots, n$. Then $$ \phi \left(\sum_{j=1}^n \lambda_j \Delta_j \right) (a) = \phi \left(\sum_{j=1}^n \lambda_j D^{j}_{a,b,\phi} \right) (a)$$ and $$ \phi \left(\sum_{j=1}^n \lambda_j \Delta_j \right) (b) = \phi \left(\sum_{j=1}^n \lambda_j D^{j}_{a,b,\phi} \right) (b),$$ which proves the statement.\smallskip

$(b)$ Suppose $\Delta : A\to A$ is a weak-2-local derivation. Given $a,b\in A$, $\phi \in A^*$, we consider the mapping $\phi^*\in A^*$ defined by $\phi^* (a) := \overline{\phi(a^*)}$ ($a\in A$). By the assumptions on $\Delta$ there exists a derivation $D_{a,b,\phi}: A\to A$ such that $\phi^* \Delta (a^*) = \phi D_{a,b,\phi} (a^*)$ and $\phi \Delta (b^*) = \phi D_{a,b,\phi} (b^*)$. We deduce from the above that $\phi \Delta^{\sharp} (a) = \phi D_{a,b,\phi}^{\sharp} (a)$ and $\phi \Delta^{\sharp} (b) = \phi D_{a,b,\phi}^{\sharp} (b)$, which proves the statement concerning $\Delta^{\sharp}$. Since $\Delta^{\sharp\sharp} = \Delta$ the reciprocal implication is clear.\smallskip

The statement in $(c)$ follows from $(a)$ and $(b)$.
\end{proof}

\begin{remark}\label{r weak-2-local *derivation}{\rm A $^*$-derivation on a C$^*$-algebra $A$ is a derivation $D$ on $A$ satisfying $D^{\sharp} = D,$ equivalently, $D (a^*)= D(a)^*,$ for every $a\in A.$ It is easy to see that, for each $^*$-derivation $D$ on $A$, the mapping $D|_{A_{sa}} : A_{sa}\to A_{sa}$ is a Jordan derivation, that is, $D(a \circ b) = a\circ D(b) + b\circ D(a),$ for every $a,b\in A_{sa},$ where $a\circ b = \frac12 (ab + ba)$ {\rm(}we should recall that $A_{sa}$ is not, in general, an associative subalgebra of $A$, but it is always a Jordan subalgebra of $A${\rm)}.\smallskip

Conversely, if $\delta: A_{sa}\to A_{sa}$ is a Jordan derivation on $A_{sa},$ then the linear mapping $\widehat{\delta}: A\to A,$ $\widehat{\delta} (a+i b) = \delta (a) + i \delta (b)$ is a Jordan $^*$-derivation on $A$, and hence a $^*$-derivation by \cite[Theorem 6.3]{John96} and \cite[Corollary 17]{PeRu}. When $M$ is a von Neumann algebra, we can deduce, via Sakai's theorem {\rm(}cf. \cite[Theorem 4.1.6]{Sak}{\rm)} that for every Jordan derivation $\delta: M_{sa}\to M_{sa},$ there exists $z\in i M_{sa}$ satisfying $\delta (a) = [z,a],$ for every $a\in M.$}
\end{remark}

\begin{lemma}\label{l linearity on Msa} Let $\Delta$ be a weak-2-local $^*$-derivation on a C$^*$-algebra $A$. Then $\Delta(a+ i b) = \Delta(a) + i \Delta (b) = \Delta (a-i b)^*$, for every $a,b\in A_{sa}$.
\end{lemma}

\begin{proof} Let us fix $a,b\in A_{sa}.$ By assumptions, for each $\phi\in A^*$ with $\phi^* = \phi$ (that is, $\phi (a^*) = \overline{\phi(a)}$ ($a\in A$). There exists a $^*$-derivation $D_{a,a+ib,\phi}$ on $A,$ depending on $a+ib$, $a$ and $\phi$, such that $$\phi \Delta (a+ i b) = \phi D_{a,a+i b,\phi} (a+ i b) = \phi D_{a,a+i b,\phi} (a) + i \phi D_{a,a+i b,\phi} (b),$$
 and $$\phi \Delta (a) = \phi D_{a,a+ i b,\phi} (a).$$ Then $\Re\hbox{e} \phi \Delta (a+ i b) = \phi D_{a,a+i b,\phi} (a),$ for every $\phi\in A^*$ with $\phi^* = \phi$, which proves that $\Delta (a+ i b) + \Delta (a+ i b)^* = 2 \Delta (a)$. We can similarly check that $\Delta (a+ i b) - \Delta (a+ i b)^* = 2 i \Delta (b)$.
\end{proof}

It is well known that every derivation $D$ on a unital C$^*$-algebra $A$ satisfies that $D(1) =0.$ Since the elements in $A^*$ separate the points in $A$, we also get:

\begin{lemma}\label{l Delta(1)=0}
Let $\Delta$ be a weak-2-local derivation on a unital C$^*$-algebra. Then $\Delta(1)=0$.$\hfill\Box$
\end{lemma}

\begin{lemma}\label{l Delta(1-x)+Delta(x)}
Let $\Delta$ be a weak-2-local derivation on a unital C$^*$-algebra $A$. Then $\Delta(1-x)+\Delta(x)=0$, for every $x\in A$.
\end{lemma}

\begin{proof} Let $x\in A$. Given $\phi\in A^*,$ there exists a derivation $D_{x,1-x,\phi}: A\to A$, such that $\phi \Delta (x) = \phi D_{x,1-x,\phi} (x)$ and $\phi \Delta (1-x) = \phi D_{x,1-x,\phi} (1-x)$. Therefore, $$\phi (\Delta(1-x)+\Delta(x)) = \phi D_{x,1-x,\phi} (1-x +x) = 0.$$ We conclude by the Hahn-Banach theorem that $\Delta(1-x)+\Delta(x)=0$.
\end{proof}

\begin{lemma}\label{p Delta(p) p}
Let $\Delta$ be a weak-2-local derivation on a unital C$^*$-algebra, and let $p$ be a projection in $A$. Then $$p\Delta(p)p=0\quad\textrm{and}\quad (1-p)\Delta(p)(1-p)=0.$$
\end{lemma}

\begin{proof} Let $\phi$ be a functional in $A^*$ satisfying $\phi = (1-p)\phi (1-p).$ Pick a derivation $D_{p,\phi}: A\to A$ satisfying $\phi \Delta (p) = \phi D_{p,\phi} (p).$ Then $$\phi \Delta (p) = \phi \left( D_{p,\phi} (p) p +  p D_{p,\phi} (p)\right) = 0,$$ where in the last equality we applied $\phi = (1-p)\phi (1-p).$ Lemma 3.5 in \cite{BenAliPeraltaRamirez} implies that $(1-p)\Delta(p)(1-p)=0.$ Replacing $p$ with $1-p$ and applying Lemma \ref{l Delta(1-x)+Delta(x)}, we get $0 = p \Delta (1-p) p = - p \Delta (p) p .$
\end{proof}

The first statement in the following proposition is probably part of the folklore in the theory of derivations,
however we do not know an explicit reference for it.

\begin{proposition}\label{p restriction}
Let $A$ be a C$^*$-algebra, $D:A\rightarrow A$ a derivation {\rm(}respectively, a $^*$-derivation{\rm)}, and let $p$ be a projection in $A$.
Then the operator $pDp|_{pAp}:pAp\rightarrow pAp$, $x\mapsto pD(x) p$ is a derivation {\rm(}respectively, a $^*$-derivation{\rm)} on $pAp$. Consequently, if $\Delta: A\to A$ is a weak-2-local derivation {\rm(}respectively, a weak-2-local $^*$-derivation{\rm)} on $A$, the mapping $p\Delta p|_{pAp} : pAp \to pAp,$ $x\mapsto p \Delta (x) p$ is a weak-2-local derivation {\rm(}respectively, a weak-2-local $^*$-derivation{\rm)} on $pAp$.
\end{proposition}

\begin{proof} Let $T$ denote the linear mapping $pDp|_{pAp}:pAp\rightarrow pAp$, $x\mapsto p D(x) p$. We shall show that $T$ is a derivation on $pAp$.
Let $x,y\in pAp$. Since $px=xp = x$ and $py=yp= y$, we have
$$ T(xy) = p D(xy)p = p D(x) y p+ p x D(y) p $$ $$= pD(x)p y+ x pD(y)p = T(x) y+ x T(y).$$
\end{proof}

\section{weak-2-local derivations on matrix algebras}

In this section we shall study weak-2-local derivations on matrix algebras.

\begin{lemma}\label{l trace zero} Let $\Delta: M_n \to M_n$ be a weak-2-local derivation on $M_n$. Let $tr$ denote the unital trace on $M_n$. Then, $tr \Delta (x) =0$, for every $x\in M_n$.
\end{lemma}

\begin{proof} Let $x$ be an arbitrary element in $M_n$. By Sakai's theorem (cf. \cite[Theorem 4.1.6]{Sak}), every derivation on $M_n$ is inner. We deduce from our hypothesis that there exists an element $z_{x,tr}$ in $M_n$, depending on $tr$ and $x$, such that $tr \Delta(x)= tr [z_{x,\phi},x]= tr (z_{x,\phi} x - x z_{x,\phi}) =0$.
\end{proof}

The algebra $M_2$ of all 2 by 2 matrices must be treated with independent arguments.\smallskip

We set some notation. Given two elements $\xi,\eta$ in a Hilbert space $H$, the symbol $\xi\otimes \eta$ will denote the rank-one operator in $B(H)$ defined by $\xi\otimes \eta (\kappa) = (\kappa |\eta) \xi.$ We can also regard $\phi=\xi\otimes \eta$ as an element in the trace class operators (that is, in the predual of $B(H)$) defined by $\xi\otimes \eta (a) = (a(\xi)| \eta)$ ($a\in B(H)$).

\begin{theorem}\label{t w-2-local derivations on M2}
Every weak-2-local derivation on $M_2$ is linear and a derivation.
\end{theorem}

\begin{proof}
Let $\Delta$ be a weak-2-local derivation on $M_2$. To simplify notation we set $e_{ij}=\xi_i\otimes \xi_j$ for $1\leq i,j\leq 2$, where $\{\xi_1,\xi_2\}$ is a fixed orthonormal basis of $\mathbb{C}^2$. We also write $p_1=e_{11}$ and $p_2=e_{22}$. The proof is divided into several steps.\smallskip

Lemma \ref{l trace zero} shows that \begin{equation}\label{eq trace zero on the image}
tr \Delta (x) =0,
\end{equation} for every $x\in M_2$.\smallskip

\textit{Step I.} Let us write $\displaystyle \Delta(p_1)=\sum_{i,j=1}^2 \lambda_{ij} e_{ij}$, where $\lambda_{ij}\in \mathbb{C}$. For $\phi= \xi_1\otimes \xi_1\in M_2^*$ there exists an element $z= \left(
                                                        \begin{array}{cc}
                                                          z_{11} & z_{12} \\
                                                          z_{21} & z_{22} \\
                                                        \end{array}
                                                      \right)
$ in $M_2$, depending on $\phi$ and $p_1$, such that $\phi\Delta(p_1)=\phi[z,p_1]$. Since
\begin{equation}\label{[z,p1]}
    [z,p_1]= -z_{12} e_{12}+ z_{21} e_{21},
\end{equation}
we deduce that $\lambda_{11} =\phi \Delta (p_1) = \phi [z,p_1] = 0$. Since $\lambda_{11}+\lambda_{22}= tr \Delta (p_1) = 0$, we also have $\lambda_{22} =0$. Therefore,
$$\Delta(p_1)=\lambda_{12} e_{12}+\lambda_{21} e_{21}.$$

Defining $z_0:=\lambda_{21} e_{21}-\lambda_{12} e_{12}$, it follows that $\widetilde{\Delta} = \Delta-[z_0,.]$ is a weak-2-local derivation (cf. Lemma \ref{l sharp}$(a)$) which vanishes at $p_1$.  Applying Lemma \ref{l Delta(1-x)+Delta(x)}, we deduce that \begin{equation}\label{eq Delta at p1 and p2} \widetilde{\Delta} (p_1)= \widetilde{\Delta} (p_2)=0.
\end{equation}

\textit{Step II.} Let us write $\displaystyle \widetilde{\Delta} (e_{12})=\sum_{i,j=1}^2\lambda_{ij} e_{ij}$, with $\lambda_{22}= -\lambda_{11}$ (cf. \eqref{eq trace zero on the image}). For $\phi= \xi_1\otimes \xi_2 \in M_2^*$, there exists an element $z= \left(
                                                        \begin{array}{cc}
                                                          z_{11} & z_{12} \\
                                                          z_{21} & z_{22} \\
                                                        \end{array}
                                                      \right)$ in $M_2$, depending on $\phi$ and $e_{12}$, such that $\phi\widetilde{\Delta}(e_{12})=\phi[z,e_{12}]$. Since
  \begin{equation}\label{[z,p12]}
    [z,e_{12}]=-z_{21}p_1+(z_{11}-z_{22})e_{12}+z_{21}p_2,
  \end{equation} we see that $\lambda_{21}=0$.\smallskip

For $\phi= \xi_1\otimes \xi_1 - \xi_1\otimes \xi_2 \in M_2^*$,  there exists an element $z= \left(
                                                        \begin{array}{cc}
                                                          z_{11} & z_{12} \\
                                                          z_{21} & z_{22} \\
                                                        \end{array}
                                                      \right)$ in $M_2$, depending on $\phi$, $p_1$ and $e_{12}$, such that
$\phi\widetilde{\Delta}(p_1)=\phi[z,p_1]$ and $\phi\widetilde{\Delta}(e_{12})=\phi[z,e_{12}]$. The identities \eqref{[z,p1]} and \eqref{[z,p12]} (and \eqref{eq Delta at p1 and p2}) imply that $\lambda_{11}=-z_{21}$ and $0= -z_{21}$, and hence $\lambda_{11}=0$. Therefore, there exists a complex number $\delta$ satisfying
  \begin{equation}\label{Delta(p12)}
    \widetilde{\Delta}(e_{12})=\delta e_{12} = \left[z_1, e_{12} \right],
  \end{equation} where $z_1 = \left(
                                \begin{array}{cc}
                                  \delta &0 \\
                                  0 & 0 \\
                                \end{array}
                              \right)
  $. We observe that $[z_1,\lambda p_1 + \mu p_2] =0$, for every $\lambda,\mu\in \mathbb{C}$. Thus, the mapping $\widehat{\Delta} = \widetilde{\Delta}- [z_1,.] =\Delta -[z_0,.]-[z_1,.]$ is a weak-2-local derivation satisfying \begin{equation}\label{eq widehatDelta} \widehat{\Delta} (e_{12}) =  \widehat{\Delta} ( p_1 ) = \widehat{\Delta} ( p_2 )=0.
  \end{equation}

\textit{Step III.} Let us write $\displaystyle \widehat{\Delta}(e_{21})=\sum_{i,j=1}^2\lambda_{ij} e_{ij}$, with $\lambda_{11}= -\lambda_{22}$ (see Lemma \ref{l trace zero}). For $\phi=\xi_2\otimes \xi_1\in M_2^*$, there exists an element $z= \left(
                                                        \begin{array}{cc}
                                                          z_{11} & z_{12} \\
                                                          z_{21} & z_{22} \\
                                                        \end{array}
                                                      \right)$ in $M_2$, depending on $\phi$ and $e_{21}$, such that $\phi\widehat{\Delta}(e_{21})=\phi[z,e_{21}]$. Since
  \begin{equation}\label{[z,p21]}
    [z,e_{21}]=z_{12}p_1-(z_{11}-z_{22})e_{21}-z_{12}p_2,
  \end{equation} we see that $\lambda_{12}=0$.\smallskip

Take now $\phi = \xi_1\otimes \xi_1 - \xi_2\otimes \xi_1\in M_2^*$. By hypothesis, there exists an element $z= \left(
                                                        \begin{array}{cc}
                                                          z_{11} & z_{12} \\
                                                          z_{21} & z_{22} \\
                                                        \end{array}
                                                      \right)$ in $M_2$, depending on $\phi$, $p_1$ and $e_{21}$, such that $\phi\widehat{\Delta}(p_1)=\phi[z,p_1]$ and $\phi\widehat{\Delta}(e_{21})=\phi[z,e_{21}]$.
We deduce from \eqref{[z,p1]}, \eqref{[z,p21]} and \eqref{eq widehatDelta} that  $z_{12} = \lambda_{11}$ and $z_{12} = 0,$ which gives $\lambda_{11}=0$.\smallskip

For $\phi = \xi_2\otimes \xi_1- \xi_1\otimes \xi_2\in M_2^*$, there exists an element $z= \left(
                                                        \begin{array}{cc}
                                                          z_{11} & z_{12} \\
                                                          z_{21} & z_{22} \\
                                                        \end{array}
                                                      \right)$ in $M_2$, depending on $\phi$, $e_{12}$ and $e_{21}$,
such that $\phi\widehat{\Delta}(e_{12})=\phi[z,e_{12}]$ and $\phi\widehat{\Delta}(e_{21})=\phi[z,e_{21}]$. We apply
\eqref{[z,p12]}, \eqref{[z,p21]} and \eqref{eq widehatDelta} to obtain $-\lambda_{21} = z_{11}- z_{22}$ and $0= \phi\widehat{\Delta}(e_{12}) = z_{11}- z_{22},$ which proves that $\lambda_{21}=0$. Therefore
  \begin{equation}\label{Delta(p21)}
    \widehat{\Delta}(e_{21})=0.
  \end{equation}

We shall finally prove that $\widehat{\Delta}\equiv 0$, and consequently $\Delta=  [z_0,.]+[z_1,.]$ is a linear mapping and a derivation. \smallskip

\textit{Step IV.} Let us fix $\alpha,\beta \in \mathbb{C}$. We write $\displaystyle \widehat{\Delta}(\alpha e_{12}+\beta e_{21})=\sum_{i,j=1}^2 \lambda_{ij} e_{ij}$, where $\lambda_{11}= -\lambda_{22}$. For $\phi= \xi_2\otimes \xi_1\in M_2^*$, there exists an element $z= \left(
                                                        \begin{array}{cc}
                                                          z_{11} & z_{12} \\
                                                          z_{21} & z_{22} \\
                                                        \end{array}
                                                      \right)$ in $M_2$, depending on $\phi$, $e_{12}$ and $\alpha e_{12}+ \beta e_{21}$, such that $\phi\widehat{\Delta}(e_{12})=\phi[z,e_{12}]$ and $\phi\widehat{\Delta}(\alpha e_{12}+ \beta e_{21})=\phi[z,\alpha e_{12}+ \beta e_{21}]$. Since
\begin{equation}\label{[z,ap12+p21]}
   [z,\alpha e_{12}+\beta e_{21}] = (\beta z_{12}-\alpha z_{21}) p_1 +\alpha (z_{11}-z_{22}) e_{12}
\end{equation} $$+ \beta (z_{22}-z_{11}) e_{21}+(\alpha z_{21}- \beta z_{12}) p_2,$$ we have $\lambda_{12}= \alpha (z_{11}-z_{22})$. Now, the identities \eqref{[z,p12]} and \eqref{eq widehatDelta} imply $z_{11}-z_{22}=0$, and hence $\lambda_{12}=0$.\smallskip

For $\phi=\xi_1\otimes \xi_2\in M_2^*$ there exists an element $z= \left(
                                                        \begin{array}{cc}
                                                          z_{11} & z_{12} \\
                                                          z_{21} & z_{22} \\
                                                        \end{array}
                                                      \right)$ in $M_2$,
depending on $\phi$, $e_{21}$ and $\alpha e_{12}+\beta e_{21}$, such that $\phi\widehat{\Delta}(e_{21})=\phi[z,e_{21}]$ and $\phi\widehat{\Delta}(\alpha e_{12}+\beta e_{21})=\phi[z,\alpha e_{12}+ \beta e_{21}]$. We deduce from \eqref{[z,p21]}, \eqref{[z,ap12+p21]} and \eqref{Delta(p21)}, that $\lambda_{21} = \beta(z_{22}-z_{11})$ and $z_{22}-z_{11}=0$, witnessing that $\lambda_{21}=0$.\smallskip

For $\phi = \xi_1\otimes \xi_1+\beta \xi_2\otimes \xi_1+\alpha \xi_1\otimes \xi_2\in M_2^*$ there exists an element $z$ in $M_2$, depending on $\phi$, $p_1$ and $\alpha e_{12}+ \beta e_{21}$, such that $\phi\widehat{\Delta}(p_1)=\phi[z,p_1]$ and $\phi\widehat{\Delta}(\alpha e_{12}+\beta e_{21})=\phi[z,\alpha e_{12}+ \beta e_{21}]$. It follows from \eqref{[z,ap12+p21]} and \eqref{[z,p1]} that $\lambda_{11}+\beta \lambda_{12}+\alpha \lambda_{21}=\beta z_{12}-\alpha z_{21}$, and $-\beta z_{12}+\alpha z_{21}=0$, which implies that $\lambda_{11}=0$, and hence
  \begin{equation}\label{Delta(ap12+p21)}
    \widehat{\Delta}(\alpha e_{12}+\beta e_{21})=0,
  \end{equation} for every $\alpha, \beta\in \mathbb{C}$.\smallskip

\textit{Step V.} In this step we fix two complex numbers $t,\alpha\in\mathbb{C}$, and we write $\displaystyle \widehat{\Delta}(t p_1+ \alpha e_{12} )=\sum_{i,j=1}^2\lambda_{ij} e_{ij}$, with $\lambda_{11}= -\lambda_{22}$. Applying that $\widehat{\Delta}$ is a weak-2-local derivation with $\phi= \xi_1\otimes \xi_1\in M_2^*$, $e_{12}$ and $t p_1+ \alpha e_{12}$, we deduce from the identity
\begin{equation}\label{[z,tp1+ap12]}
    [z,tp_1+\alpha e_{12}] =-\alpha z_{21} p_1+(\alpha z_{11}- t z_{12}-\alpha z_{22}) e_{12}
  \end{equation} $$+t z_{21} e_{21}+\alpha z_{21} p_2,$$ combined with \eqref{[z,p12]} and \eqref{eq widehatDelta}, that $-\alpha z_{21} = \lambda_{11}$, and $z_{21} =0,$ and hence $\lambda_{11}=0$.\smallskip

Repeating the above arguments with $\phi = \xi_1\otimes \xi_2\in M_2^*$, $p_{1}$ and $t p_1+\alpha e_{12}$, we deduce from \eqref{[z,p1]}, \eqref{[z,tp1+ap12]} and \eqref{eq widehatDelta}, that $\lambda_{21}= t z_{21}$ and $z_{21} =0$,  which proves that $\lambda_{21}=0$.\smallskip

A similar reasoning with $\phi= t \xi_1\otimes \xi_1 -\alpha \xi_2\otimes \xi_1\in M_2^*$, $\alpha e_{12} +\alpha e_{21}$ and $t p_1+ \alpha e_{12}$, gives, via \eqref{[z,ap12+p21]}, \eqref{Delta(ap12+p21)}, and \eqref{[z,tp1+ap12]}, that $t \lambda_{11}-\alpha \lambda_{12}= t \alpha z_{12} - t \alpha z_{21} -\alpha^2 z_{11} +\alpha^2 z_{22}$ and $t \alpha z_{12} - t \alpha z_{21} -\alpha^2 z_{11} +\alpha^2 z_{22} =0$. Therefore $\alpha \lambda_{12}=0$ and
\begin{equation}\label{Delta(tp1+a p12)}
    \widehat{\Delta}(t p_1+ \alpha e_{12} )=0,
\end{equation} for every $t,\alpha\in \mathbb{C}$.\smallskip

A similar argument shows that
\begin{equation}\label{Delta(tp1+b p21)}
    \widehat{\Delta}(t p_1+ \beta e_{21} )=0,
\end{equation} for every $t,\beta\in \mathbb{C}$.\smallskip

\textit{Step VI.} In this step we fix $t,\alpha,\beta\in\mathbb{C}$, and we write $$\displaystyle \widehat{\Delta}(tp_1+\alpha e_{12}+\beta e_{21})=\sum_{i,j=1}^2\lambda_{ij} e_{ij},$$ with $\lambda_{11}= -\lambda_{22}$. Applying that $\widehat{\Delta}$ is a weak-2-local derivation with $\phi= \alpha \xi_1\otimes \xi_2 +\beta \xi_2\otimes \xi_1\in M_2^*$, $p_1$ and $t p_1+\alpha e_{12}+ \beta e_{21}$, we deduce from the identity
\begin{equation}\label{[z,tp1+a p12+b p21]}
    [z,tp_1+\alpha e_{12}+\beta e_{21}] = (\beta z_{12}-\alpha z_{21})p_1+(\alpha z_{11}-\alpha z_{22}-t z_{12})e_{12}
\end{equation} $$+ (\beta z_{22}-\beta z_{11}+t z_{21})e_{21}+(\alpha z_{21}-\beta z_{12})p_2,$$
combined with \eqref{[z,p1]} and \eqref{eq widehatDelta}, that $\beta \lambda_{12} + \alpha \lambda_{21}= t (\alpha z_{21}- \beta z_{12})$ and $\alpha z_{21}- \beta z_{12}=0$, which gives $\beta \lambda_{12} + \alpha \lambda_{21}=0$.\smallskip

Repeating the above arguments with $\phi = t \xi_1\otimes \xi_1 +\alpha \xi_1\otimes \xi_2\in M_2^*$, $e_{21}$ and $t p_1+ \alpha e_{12}+\beta e_{21}$, we deduce from \eqref{[z,p21]}, \eqref{Delta(p21)} and \eqref{[z,tp1+a p12+b p21]}, that $t \lambda_{11}+\alpha \lambda_{21}= \beta (t z_{12}+ \alpha z_{22}-\alpha z_{11})$, and $t z_{12}+ \alpha z_{22}-\alpha z_{11}=0$ and hence $t \lambda_{11}+\alpha \lambda_{21}=0$.\smallskip

A similar reasoning with $\phi= t \xi_1\otimes \xi_1 + \beta \xi_2\otimes \xi_1\in M_2^*$, $e_{12}$ and $t p_1+ \alpha e_{12}+\beta e_{21}$, gives, via \eqref{[z,p12]}, \eqref{eq widehatDelta} and \eqref{[z,tp1+a p12+b p21]}, that $t \lambda_{11}+ \beta \lambda_{12}=\alpha(-t z_{21} +\beta z_{11} -\beta z_{22})$ and $-t z_{21} +\beta z_{11} -\beta z_{22}=0$. Therefore $t \lambda_{11}+ \beta \lambda_{12}=0$. The equations  $\beta \lambda_{12} + \alpha \lambda_{21}=0$, $t \lambda_{11}+\alpha \lambda_{21}=0$, and $t \lambda_{11}+ \beta \lambda_{12}=0$ imply that $t \lambda_{11}= \beta \lambda_{12}=\alpha \lambda_{21}=0$, which, combined with \eqref{Delta(ap12+p21)}, \eqref{Delta(tp1+a p12)} and \eqref{Delta(tp1+b p21)}, prove that \begin{equation}\label{Delta(tp1+a p12+b p21)}
    \widehat{\Delta}(t p_1+\alpha e_{12}+\beta e_{21})=0,
\end{equation} for every $t,\alpha,\beta\in \mathbb{C}$.\smallskip

Finally, since $$[z,t p_1+\alpha e_{12}+\beta e_{21}+ s p_2]=[z,(t-s) p_1+\alpha e_{12}+\beta e_{21}],$$ for every $z\in M_2$, it follows from the fact that $\widehat{\Delta}$ is a weak-2-local derivation, \eqref{Delta(tp1+a p12+b p21)}, and the Hahn-Banach theorem that $$\widehat{\Delta}(t p_1+\alpha e_{12}+\beta e_{21}+ s p_2)= \widehat{\Delta}((t-s) p_1+\alpha e_{12}+\beta e_{21})=0,$$ for every $t,s,\alpha,\beta\in\mathbb{C}$, which concludes the proof.\end{proof}

The rest of this section is devoted to the study of weak-2-local derivations on $M_n$. For later purposes, we begin with a strengthened version of Lemma \ref{p Delta(p) p}.

\begin{lemma}\label{l almost orthogonality} Let $\Delta: M\to M$ be a weak-2-local projection on a von Neumann algebra $M$. Suppose $p,q$ are orthogonal projections in $M$, and $a$ is an element in $M$ satisfying $p a = ap = q a = aq =0$. Then the  identities:
$$p\Delta (a+ \lambda p +\mu q) q =  p\Delta (\lambda p +\mu q) q, \hbox{ and, } p\Delta (a+ \lambda p ) p = \lambda p\Delta ( p ) p=0,$$ hold for every $\lambda,\mu\in \mathbb{C}.$ Furthermore, if $b$ is another element in $M$, we also have $$q \Delta (b+ \lambda p) q= q \Delta (b) q, \hbox{ and } q \Delta (q b q+ \lambda q) q= q \Delta (q b q) q.$$
\end{lemma}

\begin{proof} Clearly, $p+q$ is a projection in $M$. Let $\phi$ any functional in $M_*$ satisfying $\phi = (p+q) \phi (p+q).$ By hypothesis, there exists an element $z_{\phi,\lambda p +\mu q,a+\lambda p+\mu q}\in M$, depending on $\phi,$ $\lambda p+\mu q,$ and $a+\lambda p+\mu q$, such that $$\phi \Delta (a+\lambda p +\mu q) = \phi [z_{\phi,\lambda p+\mu q,a+\lambda p+\mu q} ,a+\lambda p +\mu q],$$ and $$\phi \Delta ( \lambda p +\mu q) = \phi [z_{\phi,\lambda p+\mu q,a+\lambda p+\mu q},\lambda p+\mu q].$$
Since
$$ \phi [z_{\phi,\lambda p+\mu q,a+\lambda p+\mu q} ,a+\lambda p +\mu q] = \phi [z_{\phi,\lambda p+\mu q,a+\lambda p+\mu q},\lambda p+\mu q],$$
we deduce that $\phi (\Delta (a+\lambda p +\mu q)  - \Delta ( \lambda p +\mu q ) )=0$, for every $\phi \in M_*$ with $\phi = (p+q) \phi (p+q).$ Lemma 2.2 in \cite{BenAliPeraltaRamirez} implies that $$(p+q) \Delta (a+\lambda p +\mu q) (p+q) =  (p+q) \Delta (\lambda p +\mu q) (p+q).$$ Multiplying on the left by $p$ and on the right by $q$, we get $p\Delta (a+ \lambda p +\mu q ) q =  p\Delta ( \lambda p +\mu q ) q$. The other statements follow in a similar way.
\end{proof}

\begin{proposition}\label{p finite additivity and linearity on projections} Let $\Delta: M \to M$ be a weak-2-local derivation on a  von Neumann algebra $M$. Then for every family $\{p_1,\ldots,p_n\}$ of mutually orthogonal projections in $M$, and every $\lambda_1,\ldots, \lambda_n$ in $\mathbb{C},$ we have $$\Delta \left(\sum_{j=1}^{n} \lambda_j p_j \right) = \sum_{j=1}^{n} \lambda_j \Delta (p_j).$$
\end{proposition}

\begin{proof} Let $p_1,\ldots,p_n$ be mutually orthogonal projections in $M$. First, we observe that, by the last statement in Lemma \ref{l almost orthogonality}, for any $1\leq i,k\leq n,\ i\ne k,$ we have
$$(p_i+p_k) \Delta(\lambda_ip_i+\lambda_kp_k)(p_i+p_k) $$ $$= (p_i+p_k) \Delta((\lambda_i-\lambda_k)p_i + \lambda_k(p_i+p_k))(p_i+p_k) = (p_i+p_k) \Delta((\lambda_i-\lambda_k)p_i)(p_i+p_k) $$ $$=(p_i+p_k) \lambda_i\Delta(p_i)(p_i+p_k)-(p_i+p_k) \lambda_k\Delta(p_i)(p_i+p_k) $$
$$=(p_i+p_k) \lambda_i\Delta(p_i)(p_i+p_k)-(p_i+p_k) \lambda_k\Delta(p_i+ p_k -p_k)(p_i+p_k) $$
$$= (p_i+p_k) \lambda_i\Delta(p_i)(p_i+p_k) + (p_i+p_k) \lambda_k\Delta(p_k)(p_i+p_k),$$
where the last step is obtained by another application of Lemma \ref{l almost orthogonality}. Multiplying on the left hand side by $p_i$ and on the right hand side by $p_k$ we obtain:
\begin{equation}\label{eq p_i+p_j}
  p_i \Delta(\lambda_ip_i\!+\!\lambda_kp_k)p_k = \lambda_ip_i\Delta(p_i)p_k\! +\! \lambda_kp_i \Delta(p_k)p_k,\ (1\leq i,k\leq n,\ i\ne k).
\end{equation}

Let us write $\displaystyle r=1-\sum_{j=1}^n p_j$ and
\begin{equation}\label{eq 0301 2}
  \Delta\left(\sum_{j=1}^n\lambda_jp_j\right) = r\Delta\left(\sum_{j=1}^n\lambda_jp_j\right)r + \sum_{i=1}^n \left(p_i\Delta\left(\sum_{j=1}^n\lambda_jp_j\right)r \right)
\end{equation}
$$+\sum_{k=1}^n \left( r\Delta\left(\sum_{j=1}^n\lambda_jp_j\right)p_k \right)+\sum_{i,k=1}^n \left(p_i\Delta\left(\sum_{j=1}^n\lambda_jp_j\right)p_k\right).$$

Applying Lemma \ref{l almost orthogonality} we get: $\displaystyle r\Delta\left(\sum_{j=1}^n\lambda_jp_j\right)r=0.$ Given $1\leq i\leq n,$ the same Lemma \ref{l almost orthogonality} implies that
\begin{equation}\label{eq 0301 3}
  p_i\Delta\left(\sum_{j=1}^n\lambda_jp_j\right)r=p_i\Delta\left(\sum_{j=1,j\ne i}^n\lambda_jp_j+\lambda_ip_i\right)r = \lambda_ip_i\Delta(p_i)r,
\end{equation}
and similarly
\begin{equation}\label{eq 0301 4}
  r\Delta\left(\sum_{j=1}^n\lambda_jp_j\right)p_i = \lambda_ir\Delta(p_i)p_i,\quad \textrm{and} \quad p_i\Delta\left(\sum_{j=1}^n\lambda_jp_j\right)p_i=0.
\end{equation}

Given $1\leq i,k\leq n,$  $i\ne k$, Lemma \ref{l almost orthogonality} proves that
\begin{equation}\label{eq 0301 5}
  p_i\Delta\left(\sum_{j=1}^n\lambda_jp_j\right)p_k = p_i\Delta\left(\sum_{j=1,j\ne i,k}^n\lambda_jp_j + \lambda_ip_i + \lambda_kp_k\right)p_k
\end{equation}
$$= p_i\Delta(\lambda_ip_i + \lambda_kp_k)p_k = \hbox{(by \eqref{eq p_i+p_j})}= \lambda_ip_i\Delta(p_i)p_k + \lambda_kp_i\Delta(p_k)p_k.$$

We also have: \begin{equation}\label{eq 03032015} \Delta (p_j) = p_j \Delta(p_j) r + r \Delta (p_j) p_j + \sum_{k=1}^{n} p_j \Delta(p_j) p_k + p_k \Delta(p_j) p_j.
\end{equation}

Finally, the desired statement follows from \eqref{eq 0301 2}, \eqref{eq 0301 3}, \eqref{eq 0301 4}, \eqref{eq 0301 5}, and \eqref{eq 03032015}.
\end{proof}

\begin{corollary}\label{c to propo additivity on orthogonal projections} Let $\Delta: M \to M$  be a weak-2-local derivation on a von Neumann algebra. Suppose $a$ and $b$ are elements in $M$ which are written as finite linear complex linear combinations $\displaystyle a= \sum_{i=1}^{m_1} \lambda_{i} p_i$ and $\displaystyle b= \sum_{j=1}^{m_2} \mu_{j} q_j$, where $p_1,\ldots, p_{m_1},q_1,\ldots,q_{m_2}$ are mutually orthogonal projections {\rm(}these hypotheses hold, for example, when $a$ and $b$ are algebraic orthogonal self-adjoint elements in $M${\rm)}. Then $\Delta (a+ b) = \Delta (a) + \Delta (b).$ $\hfill\Box$
\end{corollary}

Let $\Delta: M\to M$ be a weak-2-local derivation on a von Neumann algebra. Let $\mathcal{P}(M)$ denote the set of all projections in $M$. Proposition \ref{p finite additivity and linearity on projections} asserts that the mapping $\mu:\mathcal{P}(M)\rightarrow M$, $p\mapsto \mu(p):=\Delta(p)$ is a finitely additive measure on $\mathcal{P}(M)$ in the usual terminology employed around the Mackey-Gleason theorem (cf. \cite{BuWri92}, \cite{Doro}, and \cite{Shers2008}), i.e. $\mu (p+q) = \mu(p) + \mu (q)$, whenever $p$ and $q$ are mutually orthogonal projections in $M$. Unfortunately, we do not know if, the measure $\mu$ is, in general, bounded. \smallskip

We recall some other definitions. Following the usual nomenclature in \cite{Doro,Shers2008,AyuKuday2014} or \cite{KOPR2014}, a scalar or signed measure $\mu: \mathcal{P} (M) \to \mathbb{C}$ is said to be \emph{completely additive} or a \emph{charge} if
\begin{equation}\label{eq completely additive}
\mu\left(\sum\limits_{i\in I} p_i\right) =\sum\limits_{i\in I}\mu(p_i)
\end{equation} for every family  $\{p_i\}_{i\in I}$ of mutually orthogonal projections in $M.$ Where $\displaystyle \sum\limits_{i\in I} p_i$ is the sum of the family $(p_i)$ with respect to the weak$^*$-topology of $M$ (cf. \cite[Page 30]{Sak}), and in the right hand side, the convergence of an uncountable family is understood as summability in the usual sense. The main results in \cite{Doro1990} shows that if $M$ is a von Neumann algebra of type $I$ with no type $I_n$ ($n<\infty$) direct summands and $M$ acts on a separable Hilbert space, then any completely additive measure on $\mathcal{P}(M)$ is bounded. The conclusion remains true when $M$ is a continuous von Neumann algebra (cf. \cite{Doro}, see also \cite{Shers2008}). The next remark shows that is not always true when $M$ is a type $I_n$ factor with $2\leq n<\infty$. \smallskip

\begin{remark}\label{r non boundedness of completely additive measure on PMn}{\rm
In $M_n$ (with $2\leq n<\infty$) every family of non-zero pairwise orthogonal projections is necessarily finite so, every finitely additive measure $\mu$ on $\mathcal{P}(M_n)$ is completely additive. However, the existence of unbounded finitely additive measures on $\mathcal{P}(M_n)$ is well known in literature, see, for example, the following example inspired by \cite{Wright1998}. By the arguments at the end of the proof of \cite[Theorem 3.1]{Wright1998}, we can always find a countable infinite set of projections $\{p_n : n\in \mathbb{N}\}$ which is linearly independent over $\mathbb{Q}$, and we can extend it, via Zorn's lemma, to a Hamel base $\{ z_{j} : j\in \Lambda \}$ for $(M_n)_{sa}$ over $\mathbb{Q}$. Clearly, every element in $M_n$ can be written as a finite $\mathbb{Q}\oplus i \mathbb{Q}$-linear combination of elements in this base. If we define a $\mathbb{Q}\oplus i \mathbb{Q}$-linear mapping $\mu : M_n \to \mathbb{C}$ given by $$\mu (z_j) := \left\{\begin{array}{cc}
            (n+1), & \hbox{if $z_j = p_n$ for some natural number $n$};\\
            0, & \hbox{otherwise.}
          \end{array}\right.
 $$ Clearly, $\mu|_{\mathcal{P} (M_n)} : \mathcal{P} (M_n) \to \mathbb{C}$ is an unbounded completely additive measure.}
\end{remark}

We shall show later that the pathology exhibited in the previous remark cannot happen for the measure $\mu$ determined by a weak-2-local $^*$-derivation on $M_n$ (cf. Proposition \ref{p boundedness of Delta}). The case $n=2$ was fully treated in Theorem \ref{t w-2-local derivations on M2}.

\begin{proposition}\label{p additivity on M3} Let $\Delta: M_3\to M_3$ be a weak-2-local $^*$-derivation. Suppose $p_1,p_2, p_3$ are mutually orthogonal minimal projections in $M_3$, $e_{k3}$ is the unique minimal partial isometry in $M_3$ satisfying $e_{k3}^* e_{k3} = p_{3}$ and $e_{k3} e_{k3}^* = p_k$ {\rm(}$k=1,2${\rm)}. Let us assume that $\Delta (p_j) = \Delta (e_{k3})=0$, for every $j=1,2,3,$ $k=1,2.$ Then $$\displaystyle \Delta \left(\sum_{j=1}^{3} \lambda_{j} p_j + \sum_{k=1}^{2} \mu_k e_{k3}\right) =0,$$ for every $\lambda_1,\lambda_2,\lambda_3,\mu_1,\mu_2$ in $\mathbb{C}$.
\end{proposition}

\begin{proof} Along this proof we write $M= M_3$. For each $i\neq j$ in $\{1,2, 3\}$, we shall denote by $e_{ij}$ the unique minimal partial isometry in $M$ satisfying $e_{ij}^* e_{ij} = p_j$ and $e_{ij} e_{ij}^* = p_i$, while the symbol $\phi_{ij}$ will denote the unique norm-one functional in $M^*$ satisfying $\phi_{ij} (e_{ij})=1$. In order to simplify the notation with a simple matricial notation, we shall assume that
$$
p_1=\left(
                                                                                          \begin{array}{ccc}
                                                                                            1 & 0 & 0 \\
                                                                                            0 & 0 & 0 \\
                                                                                            0 & 0 & 0 \\
                                                                                          \end{array}
                                                                                        \right),\
p_2=\left(
                                                                                          \begin{array}{ccc}
                                                                                            0 & 0 & 0 \\
                                                                                            0 & 1 & 0 \\
                                                                                            0 & 0 & 0 \\
                                                                                          \end{array}
                                                                                        \right),\
\hbox{and}\
p_3=\left(
                                                                                          \begin{array}{ccc}
                                                                                            0 & 0 & 0 \\
                                                                                            0 & 0 & 0 \\
                                                                                            0 & 0 & 1 \\
                                                                                          \end{array}
                                                                                        \right),
$$ however the arguments do not depend on this representation.\smallskip

\emph{Step I.} We claim that, under the hypothesis of the lemma, \begin{equation}\label{eq 2 1702} \Delta (\lambda_2 p_2 + \mu_1 e_{13}) = 0 = \Delta (\lambda_1 p_1 + \mu_2 e_{23}),
\end{equation} for every $\lambda_1,\lambda_2,\mu_1,\mu_2\in \mathbb{C}$. We shall only prove the first equality, the second one follows similarly. Indeed, Corollary \ref{c to propo additivity on orthogonal projections} implies that $$\Delta (\lambda_2 p_2 + \mu_1 e_{13} \pm \overline{\mu_1} e_{31}) = \Delta (\lambda_2 p_2 ) + \Delta ( \mu_1 e_{13} \pm \overline{\mu_1} e_{31})= \Delta ( \mu_1 e_{13} \pm \overline{\mu_1} e_{31}).$$ Having in mind that $\Delta$ is a weak-2-local $^*$-derivation, we apply Lemma \ref{l linearity on Msa} to deduce that $$\Delta ( \mu_1 e_{13} \pm \overline{\mu_1} e_{31}) = \Delta ( \mu_1 e_{13}) \pm \Delta ( \mu_1 e_{13})^* =0,$$ which proves that $\Delta (\lambda_2 p_2 + \mu_1 e_{13} \pm \overline{\mu_1} e_{31}) = 0,$ for every  $\mu_1,\lambda_2\in \mathbb{C}$. Another application of Lemma \ref{l linearity on Msa} proves that $$\Delta (\lambda_2 p_2 + \mu_1 e_{13}) = \Delta (\Re\hbox{e}(\lambda_2) p_2 + \frac{\mu_1}{2} e_{13} + \frac{\overline{\mu_1}}{2} e_{31})$$ $$+ \Delta ( i \Im\hbox{m}(\lambda_2) p_2+ \frac{\mu_1}{2} e_{13} - \frac{\overline{\mu_1}}{2} e_{31}) =0.$$

\emph{Step II.} We shall prove now that \begin{equation}\label{eq 1 2402} \Delta (\lambda_2 p_2 + \mu_2 e_{23}) = 0 = \Delta (\lambda_1 p_1 + \mu_1 e_{13}),
\end{equation} for every $\lambda_1,\lambda_2,\mu_1,\mu_2\in \mathbb{C}.$ Proposition \ref{p restriction} witnesses that $$(p_2+p_3) \Delta (p_2+p_3)|_{(p_2+p_3)M (p_2+p_3)} : (p_2+p_3) M(p_2+p_3) \to (p_2+p_3)M(p_2+p_3)$$ is a weak-2-local $^*$-derivation. Since $(p_2+p_3)M(p_2+p_3)\equiv M_2,$ Theorem \ref{t w-2-local derivations on M2} implies that $(p_2+p_3) \Delta (p_2+p_3)|_{(p_2+p_3)M (p_2+p_3)}$ is a linear $^*$-derivation. Therefore, $$(p_2+p_3) \Delta(\lambda_2 p_2 + \mu_2 e_{23})  (p_2+p_3) = \lambda_2 (p_2+p_3) \Delta(p_2 )  (p_2+p_3) $$ $$+  \mu_2 (p_2+p_3) \Delta( e_{23})  (p_2+p_3) =0 ,$$ by hypothesis. This shows that $$\Delta (\lambda_2 p_2 + \mu_2 e_{23}) = \left(
                                                                     \begin{array}{ccc}
                                                                       \omega_{11} & \omega_{12} & \omega_{13}\\
                                                                       \omega_{21} & 0 & 0 \\
                                                                       \omega_{31} & 0 & 0 \\
                                                                     \end{array}
                                                                   \right),$$ where $\omega_{ij}\in \mathbb{C}.$\smallskip

The identity
\begin{equation}\label{eq 1 0303}
\left[z, \lambda_2 p_{2}\! +\! \mu_2 e_{23}\right]\! =\! \left(
                                                           \begin{array}{ccc}
                                                             0 & \lambda_2 z_{12} & \mu_2 z_{12} \\
                                                             \!\!- \lambda_2 z_{21}\! -\! \mu_2 z_{31}\!\! & -\mu_2 z_{32} & \! \mu_2 (z_{22}\!-\!z_{33})\! -\! \lambda_2 z_{23}\!\!  \\
                                                             0 & \lambda_2 z_{32} & \mu_2 z_{32}  \\
                                                           \end{array}
                                                         \right),
\end{equation}
holds for every matrix $z\in M$. Taking the functional $\phi_{11}$ (respectively $\phi_{31}$) in $M^*$, we deduce, via the weak-2-local property of $\Delta$ at $\lambda_2 p_{2} + \mu_2 e_{23}$, that $\omega_{11} = 0$ (respectively $\omega_{31} = 0$). \smallskip

The weak-2-local behavior of $\Delta$ at the points $\lambda_2 p_{2} + \mu_2 e_{23}$ and $\mu_2 e_{23}$ and the functional $\phi_{13}$, combined with \eqref{eq 1 0303}, and
\begin{equation}\label{eq 2 0303}
\left[z, \mu_2 e_{23}\right] = \left(
                                                           \begin{array}{ccc}
                                                             0  & 0 & \mu_2 z_{12} \\
                                                             - \mu_2 z_{31} & -\mu_2 z_{32}  & \mu_2 (z_{22} - z_{33}) \\
                                                             0 & 0 & \mu_2 z_{32} \\
                                                           \end{array}
                                                         \right),
\end{equation}
show that $\omega_{13}=0.$ \smallskip

The identity 
$$\left[z, - \lambda_2 p_{1} + \mu_2 e_{23}\right] = \left(
                                                           \begin{array}{ccc}
                                                             0  & \lambda_2 z_{12}  & \mu_2 z_{12} + \lambda_2 z_{13} \\
                                                             - \lambda_2 z_{21} - \mu_2 z_{31} & -\mu_2 z_{32}  & \mu_2 (z_{22} - z_{33}) \\
                                                             - \lambda_2 z_{31} & 0 & \mu_2 z_{32} \\
                                                           \end{array}
                                                         \right),
$$ combined with \eqref{eq 2 1702}, \eqref{eq 1 0303}, and the weak-2-local property of $\Delta$ at $\lambda_2 p_2 + \mu_2 e_{23}$, $- \lambda_2 p_{1} + \mu_2 e_{23}$ and the functional $\phi_{12}$ (respectively $\phi_{21}$), we obtain $\omega_{12} =0$ (respectively $\omega_{21}=0$), which means that $\Delta (\lambda_2 p_2 + \mu_2 e_{23}) = 0.$ The statement concerning $\Delta (\lambda_1 p_1 + \mu_1 e_{13})$ follows similarly.\smallskip

\emph{Step III.} We claim that \begin{equation}\label{eq 2 2402} \Delta (\lambda_1 p_1+ \lambda_2 p_2 + \mu_2 e_{23}) = 0 = \Delta (\lambda_1 p_1+ \lambda_2 p_2 + \mu_1 e_{13}) ,
\end{equation} for every $\lambda_1,\lambda_2,\mu_1,\mu_2\in \mathbb{C}.$ As before we shall only prove the first equality. Indeed, Corollary \ref{c to propo additivity on orthogonal projections} assures that $$ \Delta (\lambda_1 p_1+ p_2 + \mu_2 e_{23} + \overline{\mu_2} e_{32})  = \lambda_1 \Delta ( p_1) + \Delta ( p_2 + \mu_2 e_{23} + \overline{\mu_2} e_{32})=0,$$ where in the last equality we apply the hypothesis, \eqref{eq 1 2402} and Lemma \ref{l linearity on Msa}. Another application of Lemma \ref{l linearity on Msa} proves that $\Delta (\lambda_1 p_1+ p_2 + \mu_2 e_{23}) = 0.$ The desired statement follows from the 1-homogeneity of $\Delta$.\smallskip

\emph{Step IV.} In this step we show that \begin{equation}\label{eq 3 2402} \Delta (\lambda_1 p_1+ \lambda_2 p_2 + \mu_1 e_{13} + \mu_2 e_{23})\! = \! (1-p_3)\Delta (\lambda_1 p_1+ \lambda_2 p_2\! + \mu_1 e_{13} + \mu_2 e_{23})p_3 ,
\end{equation} for every $\lambda_1, \lambda_2, \mu_1, \mu_2\in \mathbb{C}.$ \smallskip

Since for any $z=(z_{ij})\in M$, we have
$$\left[z, \mu_1 e_{13} \right] = \left(
                                    \begin{array}{ccc}
                                      - \mu_1 z_{31} & - \mu_1 z_{32} & \mu_1(z_{11}-z_{33}) \\
                                      0  & 0  &  \mu_1 z_{21} \\
                                      0  & 0 & \mu_1 z_{31}  \\
                                    \end{array}
                                  \right),
$$ using appropriate functionals in $M^*$, we deduce, via the weak-2-local property of $\Delta$ at $w_1=\lambda_1 p_1+ \lambda_2 p_2 + \mu_1 e_{13} + \mu_2 e_{23}$ and $w_2=\lambda_1 p_1+ \lambda_2 p_2 + \mu_2 e_{23}$ ($w_1-w_2=\mu_1 e_{13}$), combined with \eqref{eq 2 2402}, that $$(p_2+p_3)\Delta (\lambda_1 p_1+ \lambda_2 p_2 + \mu_1 e_{13} + \mu_2 e_{23})(p_1+p_2)=0.$$

Considering the identity \eqref{eq 2 0303} and repeating the above arguments at the points $\lambda_1 p_1+ \lambda_2 p_2 + \mu_1 e_{13} + \mu_2 e_{23}$ and $\lambda_1 p_1+ \lambda_2 p_2 + \mu_1 e_{13},$ we show that $$p_1\Delta (\lambda_1 p_1+ \lambda_2 p_2 + \mu_1 e_{13} + \mu_2 e_{23})(p_1+p_2)=0$$
The statement in the claim \eqref{eq 3 2402} follows from the fact that $\textrm{tr}\,\Delta (\lambda_1 p_1+ \lambda_2 p_2 + \mu_1 e_{13} + \mu_2 e_{23})=0.$ \smallskip

\emph{Step V.} We claim that, \begin{equation}\label{eq Delta at peirce 1 vanishes} \Delta(\mu_1 e_{13} + \mu_2 e_{23}) =0,
\end{equation} for every $\mu_1,\mu_2$ in $\mathbb{C}$. By \eqref{eq 3 2402} $$\Delta (\mu_1 e_{13} + \mu_2 e_{23})=\left(
                                                                             \begin{array}{ccc}
                                                                               0 & 0 & \delta_{13} \\
                                                                               0 & 0 & \delta_{23} \\
                                                                               0 & 0 & 0 \\
                                                                             \end{array}
                                                                           \right),
$$ where $\delta_{ij}\in \mathbb{C}.$\smallskip

Let $\phi = \phi_{12}+\phi_{13}.$ It is not hard to see that $$\phi[z, \mu_1 e_{13} + \mu_2 e_{23}]=\phi[z, \mu_1 e_{13} + \mu_2 p_2].$$
Considering this identity, the equality in \eqref{eq 2 1702}, and the weak-2-local property of $\Delta$ at $\mu_1 e_{13} + \mu_2 e_{23}$ and $\mu_1 e_{13} + \mu_2 p_2$, we prove that $\delta_{13}=0.$ Repeating the same argument with $\phi = \phi_{21}+\phi_{23},$ $\mu_1 e_{13} + \mu_2 e_{23}$ and $\mu_1 p_{1} + \mu_2 e_{23},$ we obtain $\delta_{23}=0.$ \smallskip

\emph{Step VI.} We claim that \begin{equation}\label{eq 4 2402} \Delta (\lambda_2 p_2+ \mu_1 e_{13} + \mu_2 e_{23})=0 = \Delta (\lambda_1 p_1+ \mu_1 e_{13} + \mu_2 e_{23}),
\end{equation} for every $\lambda_1, \lambda_2, \mu_1, \mu_2\in \mathbb{C}.$ \smallskip

As in the previous steps, we shall only prove the first equality. By \eqref{eq 3 2402} $$\Delta (\lambda_2 p_2+ \mu_1 e_{13} + \mu_2 e_{23})=\left(
                                                                             \begin{array}{ccc}
                                                                               0 & 0 & \xi_{13} \\
                                                                               0 & 0 & \xi_{23} \\
                                                                               0 & 0 & 0 \\
                                                                             \end{array}
                                                                           \right),
$$ where $\xi_{ij}\in \mathbb{C}.$\smallskip

Since for any matrix $z=(z_{ij})\in M$ we have $$\phi_{13}\left[z, \lambda_2 p_2 + \mu_1 e_{13} + \mu_2 e_{23} \right]=\phi_{13}\left[z, \mu_1 e_{13} + \mu_2 e_{23} \right],$$ the
weak-2-local behavior of $\Delta$ at $\lambda_2 p_2 + \mu_1 e_{13} + \mu_2 e_{23}$ and $\mu_1 e_{13} + \mu_2 e_{23}$, combined with \eqref{eq Delta at peirce 1 vanishes}, shows that $\xi_{13}=0.$ Let $\phi=\phi_{21}+\phi_{23}$. It is easy to see that $$\phi\left[z, \lambda_2 p_2 + \mu_1 e_{13} + \mu_2 e_{23} \right]=\phi\left[z, \mu_1 p_1 + \lambda_2 p_2 + \mu_2 e_{23} \right].$$
Thus, weak-2-local property of $\Delta$ at $\lambda_2 p_2 + \mu_1 e_{13} + \mu_2 e_{23}$ and $\mu_1 p_1 + \lambda_2 p_2 + \mu_2 e_{23}$ and \eqref{eq 2 2402} show that $\xi_{23}=0,$ and hence $\Delta (\lambda_2 p_2+ \mu_1 e_{13} + \mu_2 e_{23})=0.$\smallskip

\emph{Step VII.} We shall prove that \begin{equation}\label{eq 5 2402}\Delta \left(\sum_{j=1}^{2} \lambda_{j} p_j + \sum_{k=1}^{2} \mu_k e_{k3}\right) =0,
 \end{equation} for every $\lambda_1,\lambda_2,\mu_1,\mu_2$ in $\mathbb{C}$. By \eqref{eq 3 2402} $$\Delta \left(\sum_{j=1}^{2} \lambda_{j} p_j + \sum_{k=1}^{2} \mu_k e_{k3}\right)=\left(
                                                                             \begin{array}{ccc}
                                                                               0 & 0 & \gamma_{13} \\
                                                                               0 & 0 & \gamma_{23} \\
                                                                               0 & 0 & 0 \\
                                                                             \end{array}
                                                                           \right),
$$ where $\gamma_{ij}\in \mathbb{C}.$\smallskip

Given $z=(z_{ij})\in M$ we have $$\phi_{13}\left[z, \sum_{j=1}^{2} \lambda_{j} p_j + \sum_{k=1}^{2} \mu_k e_{k3} \right]=\phi_{13}\left[z,  \lambda_{1} p_1 + \sum_{k=1}^{2} \mu_k e_{k3} \right],$$ and
$$\phi_{23}\left[z, \sum_{j=1}^{2} \lambda_{j} p_j + \sum_{k=1}^{2} \mu_k e_{k3} \right]=\phi_{23}\left[z,  \lambda_{2} p_2 + \sum_{k=1}^{2} \mu_k e_{k3} \right].$$ Then the
weak-2-local behavior of $\Delta$ at $\displaystyle \sum_{j=1}^{2} \lambda_{j} p_j + \sum_{k=1}^{2} \mu_k e_{k3}$ and $ \displaystyle \lambda_{1} p_1 + \sum_{k=1}^{2} \mu_k e_{k3}$ (respectively, $ \displaystyle \lambda_{2} p_2 + \sum_{k=1}^{2} \mu_k e_{k3}$), combined with \eqref{eq 4 2402}, imply that $\gamma_{13}=0$ (respectively, $\gamma_{23}=0$).\smallskip

Finally, for $\lambda_3\neq 0$, we have
$$\Delta \left(\sum_{j=1}^{3} \lambda_{j} p_j + \sum_{k=1}^{2} \mu_k e_{k3}\right) = \Delta \left(\lambda_3 1 + \sum_{j=1}^{2} (\lambda_{j}-\lambda_3) p_j + \sum_{k=1}^{2} \mu_k e_{k3}\right)$$
$$= \lambda_3  \Delta \left( 1 + \lambda_3^{-1} \sum_{j=1}^{2} (\lambda_{j}-\lambda_3) p_j + \lambda_3^{-1} \sum_{k=1}^{2} \mu_k e_{k3}\right)=\hbox{(by Lemma \ref{l Delta(1-x)+Delta(x)})}$$ $$= \lambda_3  \Delta \left(  \lambda_3^{-1} \sum_{j=1}^{2} (\lambda_{j}-\lambda_3) p_j + \lambda_3^{-1} \sum_{k=1}^{2} \mu_k e_{k3}\right) = \hbox{(by \eqref{eq 5 2402})} = 0,$$ for every $\lambda_1,\lambda_2, \mu_1,\mu_2$ in $\mathbb{C}$.
\end{proof}

\begin{proposition}\label{p boundedness of Delta 1} Let $\Delta: M_n\to M_n$ be a weak-2-local $^*$-derivation, where $n\in \mathbb{N}$, $2\leq n$. Suppose $p_1,\ldots, p_n$ are mutually orthogonal minimal projections in $M_n$, $q= p_1+\ldots+p_{n-1}$, $\lambda_1,\ldots, \lambda_{n}$ are complex numbers, and $a$ is an element in $M_n$ satisfying $a = q a p_n $. Then $$\Delta \left(\sum_{j=1}^{n} \lambda_j p_j + a\right) = \Delta \left(\sum_{j=1}^{n} \lambda_j p_j \right)+ \Delta (a) = \sum_{j=1}^{n} \lambda_j \Delta \left(  p_j \right) +  \Delta(a),$$ and the restriction of $\Delta $ to $q M_n p_{n} $ is linear. More concretely, there exists $w_0\in M_n,$ depending on $p_1,\ldots, p_n$, satisfying $w_0^*= -w_0$ and $$\Delta \left(\sum_{j=1}^{n} \lambda_j p_j + a\right) = \left[w_0,  \sum_{j=1}^{n} \lambda_j p_j + a\right],$$ for every $\lambda_1,\ldots, \lambda_{n}$ and $a$ as above.
\end{proposition}

\begin{proof} We shall argue by induction on $n$. The statement for $n=1$ is clear, while the case $n=2$ follows from Theorem \ref{t w-2-local derivations on M2}. We can therefore assume that $n\geq 3$.  Let us suppose that the desired conclusion is true for every $k<n$. \smallskip

As in the previous results, to simplify the notation, we write $M= M_n$. For each $i\neq j$ in $\{1,\ldots, n\}$, we shall denote by $e_{ij}$ the unique minimal partial isometry in $M$ satisfying $e_{ij}^* e_{ij} = p_j$ and $e_{ij} e_{ij}^* = p_i$. Henceforth, the symbol $\phi_{ij}$ will denote the unique norm-one functional in $M^*$ satisfying $\phi_{ij} (e_{ij})=1$. We also note that every element $a\in M$ satisfying $a = q a p_n $ writes in the form $\displaystyle a= \sum_{k=1}^{n-1} \mu_k e_{kn}$, for unique $\mu_1,\ldots, \mu_{n-1}$ in $\mathbb{C}$.\smallskip

Fix $j\in \{1,\ldots,n\}$. We observe that, for each matrix $z= (z_{ij}) \in M_n$, we have \begin{equation}\label{eq [zpj]} [z,p_j] = \sum_{k=1, k\neq j}^{n} z_{kj} e_{kj} - z_{jk} e_{jk}.
\end{equation} We deduce from the weak-2-local property of $\Delta$ that \begin{equation}\label{eq Delta pj} \Delta (p_j) = \Delta (p_j)^*  = \sum_{k=1, k\neq j}^{n} \overline{\lambda^{(j)}_{k}} e_{kj} + {\lambda^{(j)}_{k}} e_{jk},
\end{equation} for suitable  $\lambda^{(j)}_{k}\in \mathbb{C}$, $k\in \{1,\ldots,n\}\backslash\{j\}$. Given $i\neq j$, Lemma \ref{p Delta(p) p} and Proposition \ref{p finite additivity and linearity on projections} imply that $$0=(p_i+p_j) \Delta (p_i+p_j) (p_i+p_j) = (p_i+p_j)(\Delta (p_i) + \Delta (p_j)) (p_i+p_j),$$ which proves that $$\lambda_{i}^{(j)} = -\overline{\lambda_j^{(i)}},\ \ \ \ \forall i\neq j.$$ These identities show that the matrix $$z_0=-z_0^*:= \sum_{i>j} - {\lambda^{(j)}_{i}} e_{ji} + \sum_{i<j}  \overline{\lambda^{(i)}_{j}} e_{ji},$$ 
is well defined, and $\Delta (p_i) = [z_0 ,p_i]$ for every $i\in \{1,\ldots,n\}$. The mapping $\widehat{\Delta} = \Delta -[z_0,.]$ is a weak-2-local $^*$-derivation satisfying $$\widehat{\Delta} \left(\sum_{j=1}^{n} \lambda_j p_j \right) =0,$$ for every $\lambda_j\in \mathbb{C}$ (cf. Proposition \ref{p finite additivity and linearity on projections}).\smallskip

Let us fix $i_0\in \{1,\ldots,n-1\}$. It is not hard to check that the identity \begin{equation}\label{eq [z,ei0n]}  \left[z, e_{i_0n} \right]  =(z_{i_0 i_0}- z_{nn}) e_{i_0 n}+\sum_{j=1, j\neq i_0}^{n} z_{ji_0} e_{jn} -\sum_{j=1}^{n-1} z_{nj} e_{i_0j} ,
\end{equation} holds for every $z\in M.$ Combining this identity with \eqref{eq [zpj]} for $[z,p_n],$ and $[z,p_{i_0}],$ and the fact that $\widehat{\Delta}$ is a weak-2-local $^*$-derivation, we deduce, after an appropriate choosing of functionals $\phi\in M^*,$  that there exists $\gamma_{i_0 n}\in i \mathbb{R}$ satisfying $$\widehat{\Delta} (e_{i_0n}) = \gamma_{i_0 n} e_{i_0 n}, \ \ \forall i_0\in\{1,\ldots, n-1\}.$$

If we set $\displaystyle z_1:=\sum_{k=1}^{n-1} \gamma_{kn} p_k,$ then $z_1= - z_1^*,$ $$\widehat{\Delta} (e_{i_{0}n})= [z_1,e_{i_0n}],$$ for every $i_0\in\{1,\ldots, n-1\},$ and further $\displaystyle \left[z_1,\sum_{j=1}^{n} \lambda_j p_j\right]=0,$ for every $\lambda_j\in \mathbb{C}.$  Therefore, $\widetilde{\Delta} = \widehat{\Delta}-[z_1,.]$ is a weak-2-local $^*$-derivation satisfying \begin{equation}\label{eq deltatilde vanishes} \widetilde{\Delta} \left(\sum_{j=1}^{n} \lambda_j p_j\right)= \widetilde{\Delta} (e_{i_0n})=0,
 \end{equation} for every $i_0\in\{1,\ldots, n-1\}.$\smallskip

The rest of the proof is devoted to establish that $$\widetilde{\Delta} \left(\sum_{j=1}^n \lambda_j p_j + \sum_{k=1}^{n-1} \mu_k e_{kn} \right)=0,$$ for every $\mu_1,\ldots, \mu_{n-1},$ $\lambda_1,\ldots, \lambda_n$ in $\mathbb{C}$, which finishes the proof. The case $n=3$ follows from Proposition \ref{p additivity on M3}. So, henceforth, we assume $n\geq 4$. We shall split the arguments in several steps.\smallskip

\emph{Step I.} We shall first show that, for each $1\leq i_0\leq n-1,$ \begin{equation}\label{eq 1 2502} p_{i_0}\widetilde{\Delta}\left(\sum_{i=1}^n \lambda_i p_i + \mu e_{i_0n}\right)=0,
 \end{equation} for every $\lambda_1,\ldots, \lambda_n,\mu$ in $\mathbb{C}$.\smallskip

Let us pick $k\in\{1,\ldots,n-1\}$ with $k\ne i_0$. By the induction hypothesis
\begin{equation}\label{eq 2 2502} (1-p_k)\widetilde{\Delta}\left(\sum_{ i=1, i\ne k}^n \lambda_i p_i + \mu e_{i_0n}\right)(1-p_k)
 \end{equation} $$= \sum_{ i=1, i\ne k}^n  \lambda_i (1-p_k) \widetilde{\Delta} (p_i) (1-p_k) + \mu (1-p_k) \widetilde{\Delta} (e_{i_0n} ) (1-p_k)  =0.$$

Since for any $z\in M,$ the identity $$ (1-p_k)\left[z, \sum_{i=1}^n \lambda_i p_i + \mu e_{i_0n}\right](1-p_k)=(1-p_k)\left[z, \sum_{ i=1, i\ne k}^n \lambda_i p_i + \mu e_{i_0n}\right](1-p_k),$$ holds, if we take $\phi=\phi_{i_0j}$ with $j\ne k,$ we get, applying \eqref{eq 2 2502} and the weak-2-local property of $\widetilde{\Delta}$, that
$$p_{i_0}\widetilde{\Delta}\left(\sum_{i=1}^n \lambda_i p_i + \mu e_{i_0n}\right)p_j=0.\quad (1\leq j\leq n,\ j\ne k)$$
Since $4\leq n,$ we can take at least two different values for $k$ to obtain \eqref{eq 1 2502}.\smallskip

\emph{Step II.} In this step we prove that, for each $1\leq i_0\leq n-1,$ \begin{equation}\label{eq 3 2502} p_{i_0}\widetilde{\Delta}\left(\lambda p_{i_0}+\sum_{i=1}^{n-1} \mu_i e_{in} \right)p_n=0,
 \end{equation} for every $\lambda$ and $\mu_1,\ldots,\mu_{n-1}$ in $\mathbb{C}$.\smallskip

We fix $1\leq i_0\leq n-1$, and we pick $k\in\{1,\ldots,n-1\}$ with $k\ne i_0$. By the induction hypothesis, we have \begin{equation}\label{eq 4 2502} (1-p_k)\widetilde{\Delta}\left(\lambda p_{i_0}+\sum_{ i=1, i\ne k}^{n-1} \mu_i e_{in} \right)(1-p_k)
 \end{equation}
$$= \lambda (1-p_k)\widetilde{\Delta}\left( p_{i_0}\right) (1-p_k)+\sum_{ i=1, i\ne k}^{n-1 } \mu_i (1-p_k)\widetilde{\Delta}\left( e_{in} \right)(1-p_k) =0,$$ for every $\lambda$ and $\mu_1,\ldots,\mu_{n-1}$ in $\mathbb{C}$.\smallskip

Since for any $z\in M,$ the equality
$$(1-p_k)\left[z, \lambda p_{i_0}+\sum_{i=1}^{n-1} \mu_i e_{in}\right](1-p_n)=(1-p_k)\left[z, \lambda p_{i_0}+\sum_{ i=1, i\ne k}^{n-1} \mu_i e_{in}\right](1-p_n),$$ holds, we deduce from  \eqref{eq 4 2502} and the weak-2-local property of $\widetilde{\Delta}$, applied to $\phi= \phi_{i_0j}$ with $j\ne k,n$, that
$$p_{i_0}\widetilde{\Delta}\left(\lambda p_{i_0}+\sum_{i=1}^{n-1} \mu_i e_{in} \right)p_j=0,\quad (\forall 1\leq j\leq n-1,\ j\ne k).$$
By taking two different values for $k$, we see that
\begin{equation}\label{pi0 Delta(lambda+sum)(1-pn)=0}
  p_{i_0}\widetilde{\Delta}\left(\lambda p_{i_0}+\sum_{i=1}^{n-1} \mu_i e_{in} \right)(1-p_n)=0.
\end{equation}

Let $\phi_0=\sum_{j=1}^n \phi_{i_0j}$. It is not hard to see that the equality
$$\phi_0\left[z, \sum_{i=1, i\ne i_0}^{n-1} \mu_i e_{in} \right] = \phi_0\left[z, \sum_{i=1, i\ne i_0}^{n-1} \mu_i p_i \right],$$ holds for every $z\in M.$ Thus,
$$\phi_0 \left[z, \lambda p_{i_0}+\sum_{i=1}^{n-1} \mu_i e_{in} \right] = \phi_0 \left[z, \lambda p_{i_0}+\sum_{i=1, i\ne i_0}^{n-1} \mu_i p_i + \mu_{i_0} e_{i_0n}\right],$$ for every $z\in M.$ Therefore, the weak-2-local property of $\widetilde{\Delta}$ implies that
$$\phi_0\widetilde{\Delta}\left(\lambda p_{i_0}+\sum_{i=1}^{n-1} \mu_i e_{in} \right) = \phi_0\widetilde{\Delta}\left(\lambda p_{i_0}+\sum_{i=1, i\ne i_0}^{n-1} \mu_i p_i + \mu_{i_0} e_{i_0n}\right)=0,$$
where the last equality follows from \eqref{eq 1 2502}. Combining this fact with \eqref{pi0 Delta(lambda+sum)(1-pn)=0}, we get \eqref{eq 3 2502}.\smallskip

\emph{Step III.} In this final step we shall show that \begin{equation}\label{eq 5 2502} \widetilde{\Delta}\left(\sum_{i=1}^{n-1}\lambda_i p_i + \sum_{i=1}^{n-1} \mu_i e_{in} \right)=0,
 \end{equation} for every $\mu_1,\ldots, \mu_{n-1},$ $\lambda_1,\ldots, \lambda_{n-1}$ in $\mathbb{C}$.\smallskip

Let $k\in\{1,\ldots,n-1\}.$ By the induction hypothesis
$$(1-p_k)\widetilde{\Delta}\left(\sum_{i=1, i\ne k}^{n-1}\lambda_i p_i + \sum_{i=1, i\ne k}^{n-1} \mu_i e_{in} \right)(1-p_k)=0.$$
Since for any $z\in M,$ we have
$$(1-p_k)\left[z, \sum_{i=1}^{n-1}\lambda_i p_i + \sum_{i=1}^{n-1} \mu_i e_{in}\right](1-p_k-p_n)$$ $$=(1-p_k)\left[z, \sum_{i=1, i\ne k}^{n-1}\lambda_i p_i + \sum_{i=1, i\ne k}^{n-1} \mu_i e_{in}\right](1-p_k-p_n),$$
by taking $\phi=\phi_{lj},$ with $l\ne k$ and $j\ne k,n,$ we deduce, via the weak-2-local behavior of $\widetilde{\Delta}$, that
$$p_l\widetilde{\Delta}\left(\sum_{i=1}^{n-1}\lambda_i p_i + \sum_{i=1}^{n-1} \mu_i e_{in} \right)p_j=0,$$ for every $l\ne k$ and $j\ne k,n.$
Taking three different values for $k,$ we show that
\begin{equation}\label{Delta(sum+sum)(1-pn)=0}
  \widetilde{\Delta}\left(\sum_{i=1}^{n-1}\lambda_i p_i + \sum_{i=1}^{n-1} \mu_i e_{in} \right)(1-p_n)=0.
\end{equation}

Let us pick $i_0\in\{1,\ldots,n-1\}$. It is easy to check that the identity
$$p_{i_0}\left[z, \sum_{i=1}^{n-1}\lambda_i p_i + \sum_{i=1}^{n-1} \mu_i e_{in} \right]p_n = p_{i_0}\left[z, \lambda_{i_0} p_{i_0} + \sum_{i=1}^{n-1} \mu_i e_{in} \right]p_n,$$ holds for every $z\in M.$
So, taking $\phi=\phi_{i_0n}$, we deduce from the weak-2-local property of $\widetilde{\Delta}$ that
$$p_{i_0}\widetilde{\Delta}\left(\sum_{i=1}^{n-1} \lambda_i p_i + \sum_{i=1}^{n-1} \mu_i e_{in}\right)p_n=p_{i_0}\widetilde{\Delta}\left(\lambda_{i_0} p_{i_0} + \sum_{i=1}^{n-1} \mu_i e_{in}\right)p_n=0,$$
where the last equality is obtained from \eqref{eq 3 2502}. Since above identity holds for any $i_0\in\{1,\ldots,n-1\},$ we conclude that
\begin{equation}\label{(1-pn)Delta(sum+sum)pn=0}
  (1-p_n)\widetilde{\Delta}\left(\sum_{i=1}^{n-1} \lambda_i p_i + \sum_{i=1}^{n-1} \mu_i e_{in}\right)p_n=0.
\end{equation}
Now, Lemma \ref{l trace zero} implies that  $\displaystyle \textrm{tr}\,\widetilde{\Delta}\left(\sum_{i=1}^{n-1} \lambda_i p_i + \sum_{i=1}^{n-1} \mu_i e_{in}\right)=0$, which combined with \eqref{Delta(sum+sum)(1-pn)=0}, shows that
\begin{equation}\label{pn Delta(sum+sum)pn=0}
  p_n\widetilde{\Delta}\left(\sum_{i=1}^{n-1} \lambda_i p_i + \sum_{i=1}^{n-1} \mu_i e_{in}\right)p_n=0.
\end{equation}
Identities \eqref{Delta(sum+sum)(1-pn)=0}, \eqref{(1-pn)Delta(sum+sum)pn=0} and \eqref{pn Delta(sum+sum)pn=0} prove the statement in \eqref{eq 5 2502}.\smallskip

Finally, for $\lambda_n\neq 0,$ we have $$\widetilde{\Delta} \left(\sum_{j=1}^n \lambda_j p_j + \sum_{k=1}^{n-1} \mu_k e_{kn} \right)=  \widetilde{\Delta} \left(\lambda_n 1 +\sum_{j=1}^{n-1} (\lambda_j-\lambda_n) p_j + \sum_{k=1}^{n-1} \mu_k e_{kn} \right) $$ $$=\lambda_n \widetilde{\Delta} \left( 1 +\lambda_n^{-1} \sum_{j=1}^{n-1} (\lambda_j-\lambda_n) p_j + \lambda_n^{-1}\sum_{k=1}^{n-1} \mu_k e_{kn} \right)= \hbox{(by Lemma \ref{l Delta(1-x)+Delta(x)})}$$ $$= \lambda_n \widetilde{\Delta} \left( \lambda_n^{-1} \sum_{j=1}^{n-1} (\lambda_j-\lambda_n) p_j + \lambda_n^{-1}\sum_{k=1}^{n-1} \mu_k e_{kn} \right)=\hbox{(by \eqref{eq 5 2502})} =0,$$ for every $\mu_1,\ldots, \mu_{n-1},$ $\lambda_1,\ldots, \lambda_{n-1}$ in $\mathbb{C}$
\end{proof}

Our next result is a consequence of the above Proposition \ref{p boundedness of Delta 1} and Lemma \ref{l linearity on Msa}.

\begin{corollary}\label{c boundedness of Delta 1b} Let $\Delta: M_n\to M_n$ be a weak-2-local $^*$-derivation, where $n\in \mathbb{N}$, $2\leq n$. Suppose $p_1,\ldots, p_n$ are mutually orthogonal minimal projections in $M_n$, $q= p_1+\ldots+p_{n-1}$, and $a\in M_n$ satisfies $a^*=a$ and  $a = q a p_n + p_n a q$. Then $$\Delta \left(\sum_{j=1}^{n} \lambda_j p_j + a\right) = \Delta \left(\sum_{j=1}^{n} \lambda_j p_j \right)+ \Delta (a) = \sum_{j=1}^{n} \lambda_j \Delta \left(  p_j \right) +  \Delta(a),$$ for every $\lambda_1,\ldots, \lambda_{n}\in \mathbb{R}$, and the restriction of $\Delta $ to $(M_n)_{sa}\cap(q M_n p_{n} + p_n M_n q)$ is linear.
\end{corollary}

\begin{proof} Under the above hypothesis, Lemma \ref{l linearity on Msa} implies that

$$\Delta \left(\sum_{j=1}^{n} \lambda_j p_j + a\right) =  \Delta \left(\frac12 \sum_{j=1}^{n} \lambda_j p_j + q a p_n \right) + \Delta \left(\frac12  \sum_{j=1}^{n} \lambda_j p_j + q a p_n \right)^* $$
$$= \hbox{(by Proposition \ref{p boundedness of Delta 1})} =  \sum_{j=1}^{n} \lambda_j \Delta \left(  p_j \right) + \Delta \left( q a p_n \right) + \Delta \left( q a p_n \right)^*  $$
$$= \sum_{j=1}^{n} \lambda_j \Delta \left(  p_j \right) +  \Delta(q a p_n + p_n a q )= \sum_{j=1}^{n} \lambda_j \Delta \left(  p_j \right) +  \Delta(a).$$
\end{proof}

We can prove now that the measure $\mu$ on $\mathcal{P}(M_n)$ determined by a weak-2-local $^*$-derivation on $M_n$ is always bounded.

\begin{proposition}\label{p boundedness of Delta} Let $\Delta: M_n\to M_n$ be a weak-2-local $^*$-derivation, where $n\in \mathbb{N}$. Then $\Delta $ is bounded on the set $\mathcal{P}(M_n)$ of all projections in $M_{n}$.
\end{proposition}

\begin{proof} We shall proceed by induction on $n$. The statement for $n=1$ is clear, while the case $n=2$ is a direct consequence of Theorem \ref{t w-2-local derivations on M2}. We may, therefore, assume that $n\geq 3.$ Suppose that the desired conclusion is true for every $k<n.$ To simplify notation, we write $M= M_n$. We observe that, by hypothesis, $\Delta^{\sharp} = \Delta.$\smallskip

Let $p_1,\ldots, p_{n}$ be (arbitrary) mutually orthogonal minimal projections in $M.$ For each $i,j\in \{1,\ldots, n\}$, we shall denote by $e_{ij}$ the unique minimal partial isometry in $M$ satisfying $e_{ij}^* e_{ij} = p_j$ and $e_{ij} e_{ij}^* = p_i$. Henceforth, the symbol $\phi_{ij}$ will denote the unique norm-one functional in $M^*$ satisfying $\phi_{ij} (e_{ij})=1$.\smallskip

Let $q_n= p_1+\ldots+p_{n-1}.$ Proposition \ref{p restriction} implies that the mapping $$q_n\Delta q_n|_{q_nMq_n}: q_nMq_n \to q_nMq_n$$ is a weak-2-local $^*$-derivation on $q_nM q_n\equiv M_{n-1} (\mathbb{C})$. We know, by the induction hypothesis, that $q_n\Delta q_n|_{q_nMq_n}$ is bounded on the set $\mathcal{P}(q_nMq_n)$ of all projections in $q_nMq_n$.  Proposition \ref{p finite additivity and linearity on projections}, assures that $\mu : \mathcal{P} (q_nMq_n) \to q_n M q_n$, $p \mapsto q_n\Delta (p)q_n$ is a bounded, finitely additive measure. An application of the Mackey-Gleason theorem (cf. \cite{BuWri92}) proves the existence of a (bounded) linear operator $G: q_nMq_n \to q_nMq_n$ satisfying $G(p) = \mu (p) =q_n\Delta(p)q_n$, for every projection $p$ in $q_nMq_n$. Another application of Proposition \ref{p finite additivity and linearity on projections}, combined with a simple spectral resolution, shows that $q_n\Delta (a)q_n = G(a)$, for every self-adjoint element in $q_nMq_n$. Therefore, $q_n\Delta (a+b) q_n = G(a+b) = G(a) + G(b)= q_n\Delta (a) q_n + q_n\Delta (b) q_n,$ for every $a,b$ in the self-adjoint part of $q_nM q_n$.\smallskip

Now, Lemma \ref{l linearity on Msa} implies that $q_n\Delta q_n|_{q_nMq_n}$ is a $^*$-derivation on $q_nMq_n$. Therefore there exists $z_0= -z_0^*\in q_nMq_n$ such that \begin{equation}\label{eq z0 in qMq} q_n\Delta (q_naq_n) q_n  = [z_0, q_naq_n],
\end{equation} for every $a\in M$.\smallskip

Now, it is not hard to see that the identities:\begin{equation}\label{eq [z,e1n]} q_n \left[z, e_{1n} \right] q_n = -z_{n1} p_1 - \sum_{j=2}^{n-1} z_{nj} e_{1j} =- \sum_{j=1}^{n-1} z_{nj} e_{1j},
\end{equation} and
\begin{equation}\label{eq [z,ekn]} q_n \left[z, e_{kn} \right] q_n = - \sum_{j=1}^{n-1} z_{nj} e_{kj}, \ q_n \left[z, e_{nk} \right] q_n =  \sum_{j=1}^{n-1} z_{jn} e_{jk},
\end{equation}
hold for every $z\in M$,   and $1\leq k\leq n-1$ (cf. \eqref{eq [z,ei0n]}). The weak-2-local property of $\Delta$, combined with \eqref{eq [z,e1n]} and \eqref{eq [z,ekn]}, implies that $$\phi_{kl} \left( \Delta(e_{kn}) \right) = \phi_{1l} \left(\Delta(e_{1n}) \right),$$ for every $1\leq k\leq n-1$ and every $1\leq l\leq n-1$. Furthermore, for $2\leq i \leq n-1$, $1\leq j \leq n-1$ there exits $z\in M$, depending on $e_{1n}$ and $\phi_{ij}$, such that $\phi_{ij} \Delta(e_{1n})= \phi_{ij} [z,e_{1n}] = \phi_{ij} (q_n [z,e_{1n}] q_n)= \hbox{(by \eqref{eq [z,e1n]})} =0$. Therefore \begin{equation}\label{eq Delta e1n} q_n \Delta(e_{1n}) q_n = \sum_{j=1}^{n-1} \lambda_{nj} e_{1j},
\end{equation} for suitable (unique) $\lambda_{nj}$'s in $\mathbb{C}$ ($1\leq j\leq n-1$), and consequently, \begin{equation}\label{eq Delta en1} q_n \Delta(e_{n1}) q_n = q_n \Delta(e_{1n})^* q_n = \left(q_n \Delta(e_{1n}) q_n\right)^* = \sum_{j=1}^{n-1} \overline{\lambda_{nj}} e_{j1}.
\end{equation} We similarly obtain $$q_n \Delta(e_{kn}) q_n = \sum_{j=1}^{n-1} \lambda_{nj} e_{kj},$$ for every $1\leq k\leq n-1.$\smallskip

Let us define $$z_1=-z_1^*:= \sum_{j=1}^{n-1} \overline{\lambda_{nj}} e_{jn} - \lambda_{nj} e_{nj} \in p_{n}M q_n +q_n M p_{n}.$$  It is easy to check that $$q_n\Delta(e_{kn}) q_n= q_n [z_1, e_{kn}] q_n,\ q_n\Delta(e_{nk}) q_n= q_n [z_1, e_{nk}] q_n, \ \ \ \forall 1\leq k\leq n-1, $$ $$q_n[z_1,q_naq_n]q_n=0, \hbox{ and, } q_n[z_0, q_na p_{n} + p_{n} a q_n] q_n=0,$$ for every $a\in M$. Therefore  \begin{equation}\label{eq Deltahat vanishes 1} q_n {\Delta} (q_n a q_n)q_n = q_n [z_0+z_1, q_n a q_n ] q_n = q_n [z_0 , q_n a q_n ] q_n,
\end{equation}   $$q_n {\Delta} (e_{kn}) q_n= q_n [z_0+z_1, e_{kn}] q_n = q_n [z_1, e_{kn}] q_n,$$ and  $$q_n {\Delta} (e_{nk}) q_n= q_n [z_0+z_1, e_{nk}] q_n= q_n [z_1, e_{nk}] q_n,$$ for every $a\in M$, $1\leq k\leq n-1$.\smallskip

We claim that the set \begin{equation}\label{eq 1b 2502} \Big\{ q_n \Delta (b) q_n : b\in M, b^*=b, \|b\|\leq 1 \Big\}
\end{equation} is bounded. Indeed, let us take $b=b^*\in M$ with $\|b \|\leq 1$. The last statement in Lemma \ref{l almost orthogonality} shows that \begin{equation}\label{eq qDq vanishes on the orthognal} q_n  {\Delta}(b) q_n = q_n {\Delta} (q_nbq_n+ q_nbp_{n} +p_{n} b q_n + p_{n} b p_{n}) q_n
 \end{equation}$$= q_n {\Delta} (q_nbq_n+ q_nbp_{n} +p_{n} b q_n ) q_n.$$ The element $q_n b q_n$ is self-adjoint in $q_n M q_n$, so, there exist mutually orthogonal minimal projections $r_1,\ldots, r_{n-1}$ in $q_n M q_n$ and real numbers $\lambda_1,\ldots, \lambda_{n-1}$ such that $\displaystyle q_n b q_n = \sum_{j=1}^{n-1} \lambda_j r_j$ and $r_1+\ldots+r_{n-1} =q_n$. We also observe that $p_{n} b q_n + q_{n} b p_{n}$ is self-adjoint in $q_n M p_n + p_n M q_n$, thus, Corollary \ref{c boundedness of Delta 1b} implies that $$q_n  {\Delta}(b) q_n = q_n {\Delta} (q_nbq_n+ q_nbp_{n} +p_{n} b q_n ) q_n$$ $$= q_n {\Delta} (q_nbq_n) q_n + q_n {\Delta} ( q_nbp_{n} +p_{n} b q_n ) q_n $$ $$=\hbox{ (by \eqref{eq Deltahat vanishes 1}) }= q_n [z_0 , q_nbq_n] q_n + q_n [ z_1, q_nbp_{n} +p_{n} b q_n ] q_n ,$$ and hence $$\| q_n  {\Delta}(b) q_n \| \leq 2 \|z_0\|+ 2 \|z_1\|,$$ which proves the claim in \eqref{eq 1b 2502}.\smallskip

Following a similar reasoning to that given in the proof of \eqref{eq 1b 2502} we can obtain that the sets \begin{equation}\label{eq 2b 2502} \Big\{ q_1 \Delta (b) q_1 : b\in M, b^*=b, \|b\|\leq 1 \Big\}
\end{equation}  and \begin{equation}\label{eq 3b 2502} \Big\{ q_2 \Delta (b) q_2 : b\in M, b^*=b, \|b\|\leq 1 \Big\}
\end{equation} are bounded, where $q_2 = 1-p_2$ and $q_1= 1-p_1$.

The boundedness of $\Delta$ on the set $\mathcal{P} (M_n)$ of all projections in $M_n$ is a direct consequence of \eqref{eq 1b 2502}, \eqref{eq 2b 2502}, and \eqref{eq 3b 2502}.
\end{proof}

We can establish now the main result of this paper.

\begin{theorem}\label{t weak-2-local derivations on Mn are derivations} Every {\rm(}non-necessarily linear nor continuous{\rm)} weak-2-local $^*$-derivation on $M_n$ is linear and a derivation.
\end{theorem}

\begin{proof} Let $\Delta : M_n \to M_n$ be a weak-2-local $^*$-derivation. Propositions \ref{p finite additivity and linearity on projections} and \ref{p boundedness of Delta} assure that the mapping $\mu : \mathcal{P}(M_n) \to M_n$, $p\mapsto \mu (p):= \Delta (p)$ is a bounded completely additive measure on $\mathcal{P}(M_n)$. By the Mackey-Gleason theorem (cf. \cite{BuWri92}) there exists a bounded linear operator $G$ on $M_n$ such that $G(p)=\mu(p)=\Delta(p)$ for every $p\in \mathcal{P}(M_n)$.\smallskip

We deduce from the spectral resolution of self-adjoint matrices and Proposition \ref{p finite additivity and linearity on projections} that $\Delta (a) = G(a)$, for every $a\in (M_n)_{sa}$. Thus, given two self-adjoint elements $a,b$ in $M_n$, we have $$\Delta(a+b)=G(a+b)=G(a)+G(b)=\Delta(a)+\Delta(b).$$ This shows that $\Delta|_{(M_n)_{sa}}$ is a linear mapping. The linearity of $\Delta$ follows from Lemma \ref{l linearity on Msa}.
\end{proof}

\begin{corollary}\label{finite dimensional C*-algebras}
Every weak-2-local $^*$-derivation on a finite dimensional C$^*$-algebra is a derivation.
\end{corollary}

\begin{proof}
Let $A$ be a finite dimensional C$^*$-algebra. It is known that $A$ is unital and there exists a finite sequence of mutually orthogonal central projections $q_1,\cdots, q_m$ in $A$ such that $A=\bigoplus_{i=1}^m Aq_i$ and $Aq_i\cong M_{n_i}(\mathbb{C})$ for some $n_i\in\mathbb{N}$ ($1\leq i\leq m$) (cf. \cite[Page 50]{Takesaki}). \smallskip

Let $\Delta$ be a weak-2-local $^*$-derivation on $A$. Fix $1\leq i\leq m.$ By Proposition \ref{p restriction} the restriction $q_i \Delta q_i|_{Aq_i}= \Delta q_i|_{Aq_i}:q_i A q_i= A q_i\rightarrow Aq_i$ is a weak-2-local $^*$-derivation. Since $Aq_i\cong M_{n_i}(\mathbb{C})$, Theorem \ref{t weak-2-local derivations on Mn are derivations} asserts that $\Delta q_i|_{Aq_i}$ is a derivation.
\smallskip

Let $a$ be a self-adjoint element in $A q_i.$ Then $a$ writes in the form $\displaystyle a=\sum_{j=1}^{k_i} \lambda_j p_j,$ where $p_1,\cdots,p_{k_i}$ are mutually orthogonal projections in $Aq_i$ and $\lambda_1,\cdots,\lambda_{k_i}$ are real numbers. Proposition \ref{p finite additivity and linearity on projections} implies that  $$\Delta(a)=\sum_{j=1}^{k_i} \lambda_j \Delta(p_j).$$ Multiplying on the right by the central projection $1-q_i$ we get:
\begin{equation}\label{eq 2502 2}
 \Delta(a)(1-q_i)=\sum_{j=1}^{k_i} \lambda_j \Delta(p_j)(1-q_i).
\end{equation}

However, Lemma \ref{p Delta(p) p} implies that $ (1-p_j)\Delta(p_j)(1-p_j)=0,$ for every $1\leq j\leq k_i$. Since $p_j\leq q_i$ for every $j,$ we have $1-q_i\leq 1-p_j,$ which implies that $0=(1-q_i)\Delta(p_j)(1-q_i)=\Delta(p_j)(1-q_i),$  for every $1\leq j\leq k_i$. We deduce from \eqref{eq 2502 2} that $\Delta(a)=\Delta(a)q_i = q_i \Delta(a) q_i$ for every self-adjoint element $a\in Aq_i.$ Lemma \ref{l linearity on Msa} shows that the same equality holds for every  $a\in Aq_i.$ That is, $\Delta (A q_i) \subseteq A q_i $ and $\Delta|_{A q_i}$ is linear for every $1\leq i\leq m$.\smallskip

Let $(a_i)$ be a self-adjoint element in $A$, where $a_i\in A q_i$. Having in mind that every $a_i$ admits a finite spectral resolution in terms of minimal projections and $A q_i \perp A q_j$, for every $i\neq j$, it follows from Corollary \ref{c to propo additivity on orthogonal projections} (or from Proposition \ref{p finite additivity and linearity on projections}) that $\Delta ((a_i)) = (\Delta (a_i)).$  Having in mind that $\Delta|_{A q_i}$ is linear for every $1\leq i\leq m$, we deduce that $\Delta$ is additive in the self-adjoint part of $A$. Lemma \ref{l linearity on Msa} shows that $\Delta$ is actually additive on the whole of $A.$
\end{proof}


\begin{thebibliography}{22}
\bibitem{AyuKuday2014} Sh. Ayupov, K.K. Kudaybergenov, $2$-local derivations on von
Neumann algebras, to appear in \emph{Positivity}. DOI 10.1007/s11117-014-0307-3.

\bibitem{AyuKudPe2014} S. Ayupov, K. Kudaybergenov, A.M. Peralta, A survey on local and 2-local derivations on C$^*$- and von Neuman algebras, preprint 2014. arXiv:1411.2711v1.


\bibitem{BenAliPeraltaRamirez} A. Ben Ali Essaleh, A.M. Peralta, M.I. Ram{\'\i}rez, Weak-local derivations and homomorphisms on C$^*$-algebras, to appear in \emph{Linear Multilinear A.}






\bibitem{BuWri92} L.J. Bunce, J.D.M. Wright,
The Mackey-Gleason problem, \emph{Bull. Amer. Math. Soc.} \textbf{26}, 288-293 (1992).


\bibitem{BurFerGarPe2014RACSAM} M. Burgos, F.J. Fern{\' a}ndez-Polo, J.J. Garc{\'e}s, A.M. Peralta, A Kowalski-S{\l}odkowski theorem for 2-local $^*$-homomorphisms on von Neumann algebras, to appear in \emph{RACSAM}. DOI 10.1007/s13398-014-0200-8.

\bibitem{BurFerGarPe2014JMAA} M. Burgos, F.J. Fern{\' a}ndez-Polo, J.J. Garc{\'e}s, A.M. Peralta, 2-local triple homomorphisms on von Neumann algebras and JBW$^*$-triples, \emph{J. Math. Anal. Appl.} \textbf{426}, 43-63 (2015).



\bibitem{Doro1990} S.V. Dorofeev, A problem of the boundedness of a signed measure defined on the projectors of a von Neumann algebra of type I. (Russian) \emph{Izv. Vyssh. Uchebn. Zaved. Mat.}, no. {3}, 67-69 (1990); translation in \emph{Soviet Math. (Iz. VUZ)} \textbf{34}, no. 3, 77-80  (1990).

\bibitem{Doro} S.V. Dorofeev, On the problem of boundedness of a signed measure on projections of a von Neumann algebra, \emph{J. Funct. Anal.} \textbf{103}, 209-216 (1992).





\bibitem{John96} B.E. Johnson, Symmetric amenability and the nonexistence of Lie and
Jordan derivations, \emph{Math. Proc. Cambridge Philos. Soc.}
\textbf{120}, no. 3, 455-473 (1996).

\bibitem{John01} B.E. Johnson, Local derivations on C$^*$-algebras are derivations,
\emph{Trans. Amer. Math. Soc.} \textbf{353}, 313-325 (2001).


\bibitem{Kad90} R.V. Kadison, Local derivations, \emph{J. Algebra} \textbf{130}, 494-509 (1990).


\bibitem{KimKim04} S.O. Kim, J.S. Kim, Local automorphisms and derivations on $\mathbb{M}_n$, \emph{Proc. Amer. Math. Soc.} \textbf{132}, no. 5, 1389-1392 (2004).



\bibitem{KOPR2014} {K.~K.~Kudaybergenov, T.~Oikhberg, A.~M.~Peralta, B.~Russo,} \textit{2-Local triple derivations on von Neumann algebras,} arXiv:1407.3878.









\bibitem{PeRu} A. M. Peralta and B. Russo, Automatic continuity of triple derivations on
C$^*$-algebras and JB$^*$-triples, \emph{J. Algebra}  \textbf{399}, 960-977  (2014).




\bibitem{Sak60} S. Sakai, On a conjecture of Kaplansky, \emph{Tohoku Math. J.}, \textbf{12}, 31-33 (1960).

\bibitem{Sak} S. Sakai, \emph{C$^*$-algebras and W$^*$-algebras}, Springer-Verlag, Berlin 1971.

\bibitem{Sak91} S. Sakai, \emph{Operator algebras in dynamical systems. The theory of unbounded derivations in C*-algebras.} Encyclopedia of Mathematics and its Applications, 41. Cambridge University Press, Cambridge 1991.


\bibitem{Semrl97}  P. \v{S}emrl, Local automorphisms and derivations on $B(H)$, \emph{Proc. Amer. Math. Soc.} \textbf{125}, 2677-2680 (1997).

\bibitem{Shers2008} A.N. Sherstnev, \emph{Methods of bilinear forms in noncommutative theory of measure and integral}, Moscow, Fizmatlit, 2008,
256 pp.





\bibitem{Takesaki} M. Takesaki, \emph{Theory of operator algebras I}, Springer-Verlag, Berlin 1979.

\bibitem{Wright1998} J.D.M. Wright, Decoherence functionals for von Neumann quantum histories: boundedness and countable additivity, \emph{Comm. Math. Phys.} \textbf{191}, no. 3, 493-500 (1998).

\end{thebibliography}
\end{document}